\documentclass[a4paper,11pt]{amsart}

\usepackage{geometry}
\usepackage{amssymb,amsmath,amsthm,amscd}
\usepackage[ansinew]{inputenc} 
\usepackage[english]{babel}

\usepackage{newtxtext}
\usepackage{newtxmath}
\usepackage[T1]{fontenc}
\usepackage{mathrsfs}
\usepackage{amscd}
\usepackage{currfile}
\usepackage{bm}
\usepackage{graphicx}
\usepackage{multirow}
\usepackage[colorlinks]{hyperref}
\usepackage{color}

\newtheorem{theorem}{Theorem}
\newtheorem*{theorem*}{Theorem}
\newtheorem{definition}[theorem]{Definition}
\newtheorem{lemma}[theorem]{Lemma}
\newtheorem{proposition}[theorem]{Proposition}
\newtheorem{corollary}[theorem]{Corollary}
\newtheorem{example}[theorem]{Example}
\newtheorem{examples}[theorem]{Examples}

\theoremstyle{definition}
\newtheorem{remark}[theorem]{Remark}

\newcommand{\oo}{{\mathbb{O}}}
\newcommand{\hh}{{\mathbb{H}}}
\newcommand{\HH}{{\mathbb{H}}}

\newcommand{\cc}{{\mathbb{C}}}
\newcommand{\rr}{{\mathbb{R}}}
\newcommand{\zz}{{\mathbb{Z}}}
\newcommand{\nn}{{\mathbb{N}}}

\newcommand{\Eu}{{\mathbb{E}}}

\newcommand{\Aa}{\mathbb{A}}
\newcommand{\s}{{\mathbb{S}}}

\newcommand{\dB}{\overline\partial_{\B}}
\newcommand{\dM}{\overline\partial_{M}}
\newcommand{\GB}{\Gamma_\B}

\newcommand{\B}{\mathcal{B}}

\newcommand{\sr}{\mathcal{SR}}

\newcommand{\SL}{\mathcal{S}}

\newcommand{\I}{\mathcal{I}}
\newcommand{\F}{\mathcal{F}}
\newcommand{\M}{\mathcal{M}}
\newcommand{\PP}{\mathcal{P}}

\newcommand{\dibar}{\overline\partial}

\newcommand{\dif}{\vartheta_\B}
\newcommand{\difbar}{\overline{\vartheta}_\B}
\newcommand{\difM}{\vartheta_M}
\newcommand{\difbarM}{\overline{\vartheta}_M}
\newcommand\IM{\operatorname{Im}}
\newcommand\RE{\operatorname{Re}}
\newcommand\dds[1]{\partial_{x_{#1}}}

\newcommand\vs[1]{{#1}_s^\circ}
\newcommand\sd[1]{{#1}'_s}

\newcommand\Span{\operatorname{span}}

\newcommand{\ui}{\imath}

\newcommand{\OO}{\Omega}

\newcommand{\mbb}{\mathbb}

\newcommand{\R}{\mbb{R}}
\newcommand{\mc}{\mathcal}

\newcommand{\dd[3]}{\frac{\partial #2}{\partial #3}}
 
 \newcommand{\C}{\mbb{C}}
 
 \newcommand{\Q}{\mc{Q}}

 \newcommand{\bc}{\begin{center}}
 \newcommand{\ec}{\end{center}}

\newcommand{\AM}{\mathcal{AM}}

\newcommand{\DD}{\underline{D}}
\newcommand{\x}{\underline x}
\newcommand{\sC}{\underline S_\B}
\newcommand{\SA}{\underline S_A}
\newcommand{\wsC}{\widetilde{\underline S}_\B}
\newcommand{\wSA}{\widetilde{\underline S}_A}
\newcommand{\tGamma}{\widetilde\Gamma_\B}
\newcommand{\GammaA}{\widetilde\Gamma_A}

\newcommand{\SP}{\mathcal{S}}

\begin{document}

\title{Dunkl regularity over alternative $*$-algebras
}

\author{Giulio Binosi
}
\address[G.\ Binosi]{ORCID: 0000-0002-4733-6180}
\email[G.\ Binosi]{binosi@altamatematica.it}

\thanks{The first author is a member the INdAM Research group GNSAGA and was partially supported by INdAM's project SUNRISE and by the grant ``Progetto di Ricerca INdAM, Teoria delle funzioni ipercomplesse e applicazioni''}

\author{Alessandro Perotti
}
\address[A.\ Perotti]{Department of Mathematics, University of Trento, Via Sommarive 14, Trento Italy}
\address[A.\ Perotti]{ORCID: 0000-0002-4312-9504}
\email{alessandro.perotti@unitn.it}
\thanks{The second author is a member the INdAM Research group GNSAGA and was partially supported by the grant ``Progetto di Ricerca INdAM, Teoria delle funzioni ipercomplesse e applicazioni''}

\begin{abstract}
We characterise slice-regularity of functions over a real alternative *-algebra using operators that arise in Dunkl operator theory. 
We present a unifying perspective on hypercomplex analysis by defining a family of function spaces in the kernel of Dunkl-Cauchy-Riemann operators. 
Each of these function spaces, whose elements are called Dunkl-regular functions, refines Dunkl monogenic function theory and Dunkl harmonic analysis on Euclidean spaces. 
This approach allows a wide variety of hypercomplex function theories to be embedded as subcases of Dunkl monogenic function theory. This paves the way for further interactions between 
Dunkl theory and hypercomplex analysis.
\end{abstract}


\maketitle

\section{Introduction}

Dunkl operators are differential-difference operators associated with finite reflection groups. Introduced in \cite{Dunkl}, these operators appear in many areas: for example, they are used in harmonic analysis and in the study of multivariate special functions associated with root systems. We refer the reader to  \cite{DunklXu} and \cite{Rosler} for wide accounts of 
Dunkl operators theory and its applications.

In 2006 \cite{Ren_et_al}, a generalization of Clifford analysis was introduced by means of Dunkl operators. Given a finite reflection group for $\R^n$ and a root system, let $T_1,\ldots,T_n$ be the corresponding Dunkl operators. Let $\rr_n$ denote the real Clifford algebra $Cl(0,n)$ with signature $(0,n)$ and basis $e_1,\ldots,e_n$. Consider the Euclidean space $\rr^n$ embedded in the paravector space $M=\langle 1,e_1,\ldots,e_n\rangle$ of $\rr_n$ as the subspace generated by $e_1,\ldots,e_n$. The Dunkl-Dirac operator $\underline D$ on $\R_n$ is defined as $\underline D=\sum_{i=1}^ne_iT_i$. When all the multiplicities of the Dunkl operators are zero, Dunkl-Dirac analysis reduces to standard Clifford analysis. 

Clifford analysis is one of the higher dimensional generalizations of complex analysis, besides the theory of several complex variables, that have been developed in the last decades 
(see the monographs \cite{BDS,GHS} for extended accounts of this theory). 
Another higher dimensional function theory, including also polynomials and power series, was developed over the last two decades. 
The theory of slice-regular functions, also called slice analysis, was introduced in the quaternionic setting \cite{GeSt2006CR,GeSt2007Adv} and then extended to Clifford algebras, octonions,  and more generally to any real alternative *-algebra \cite{CoSaSt2009Israel,GeStRocky,AIM2011}. See \cite{GeStoSt2013IIed,Struppa2015Algebras} for reviews and extended references of this function theory. 

Dunkl-Dirac analysis and slice-regular function theory were developed independently. Recently, the paper \cite{binosi2025dunklapproachsliceregular}  provided a link between these two theories giving a characterization of slice and slice-regular functions based on the spherical Dunkl-Dirac operator and on the  Dunkl-Cauchy-Riemann operator $D=\dd{}{x_0}+\underline D$ under an appropriate choice of the multiplicity function of the Dunkl operators. This result allows to view two different function theories, the one of slice-regular functions and Clifford analysis, as subcases of the general function theory of Dunkl monogenic functions, namely the functions in the kernel of $D$.

The first purpose of the present work is to extend results of \cite{binosi2025dunklapproachsliceregular} in two directions. The first one is to provide characterizations of sliceness and slice-regularity over any hypercomplex subspace $M$ of a real alternative *-algebra $\Aa$ by means of operators arising in Dunkl theory.  In Theorem \ref{teo:poly_sliceness} we generalize, with a simpler proof, the result for polynomial functions obtained in \cite[Proposition 3.3]{binosi2025dunklapproachsliceregular}. In Theorem \ref{teo:poly_slice_regularity} we characterize polynomial slice-regularity over an alternative algebra. 
The second direction is to enlarge the class of functions to which these characterizations apply, from polynomials to any function of class $C^1$ (Theorems \ref{teo:C1_sliceness} and \ref{teo:slice_regularity}). 
This extension is more difficult to achieve and requires a stricter assumption on Dunkl multiplicities. The proof of Theorem \ref{teo:C1_sliceness} is based on the spectral properties of non-negative matrices as established by the Perron-Frobenius Theorem.
The conditions provided in \cite{binosi2025dunklapproachsliceregular} were based on Dunkl operators defined by any finite reflection group $G$ acting on $\rr^n$. Here we simplify the approach restricting the choice of $G$ to the Abelian group $\zz_2^n$. The $\zz_2^n$ Dunkl-Dirac operator has been studied extensively in \cite{DeBieGenestVinet}.

A second purpose of the present work is to define, on any hypercomplex subspace of the alternative algebra $\Aa$, a family of function spaces in the kernel of a Dunkl-Cauchy-Riemann operator $D_\PP$. 
If the hypercomplex subspace $M$ has dimension $n+1$, every space is associated in a one-to-one correspondence to a partition $\PP$ of the set $[n]=\{1,\ldots,n\}$ (see Definition \ref{def:P-Dunkl-regular}). Each of these function spaces, whose elements are called \emph{Dunkl-regular functions}, provides a refinement of Dunkl monogenic function theory on $M$ and then of Dunkl harmonic analysis on $\rr^{n+1}$. 

This approach permits to embed a large family of hypercomplex function theories existing in the literature as subcases of the theory of Dunkl monogenic functions, thus providing another unifying perspective to hypercomplex analysis.
This family of function theories includes monogenic and slice-regular functions on any hypercomplex subspace $M$ of $\Aa$ (e.g.\ monogenic functions and slice-monogenic functions on the paravector space of a Clifford algebra, quaternionic Fueter-regular and slice-regular functions, functions in the kernel of the Moisil-Teodorescu operator on the reduced quaternions \cite{MT1931}, octonionic monogenic functions and octonionic slice-regular functions), but also Clifford axially monogenic functions, octonionic slice Fueter-regular functions \cite{JinRenSabadini2020,SliceFueterRegular}, the recently defined generalized partial-slice monogenic functions of type $(p,n-p)$ on $\rr_n$ \cite{Sabadini_Xu_TAMS} and on octonions \cite{Sabadini_Xu_Octo_TAMS} and the wider class of strongly $T$-regular functions on a hypercomplex subspace of $\Aa$ \cite[\S3]{GhiloniStoppatoJGP},  \cite{GhiloniStoppato_arXiv24}, \cite{GhiloniStoppato_JMAA26}. 

All these functions are Dunkl monogenic (whence Dunkl harmonic) with respect to a Dunkl-Dirac operator $\DD_\PP$ with appropriate multiplicities. The choice of these multiplicities depends only on the partition $\PP$. 

The definition of $\DD_\PP$ uses the intermediate $\zz_2^n$ Dunkl-Dirac operators $\DD_A$, that are defined for every subset $A$ of $[n]$ (see Definition \ref{def:DiracDunklA}). They were introduced in \cite{DeBieGenestVinet} on Clifford algebras. Each of these operators, together with the left multiplication operator by the reduced variable $\x_A$, provides a realization of the orthosymplectic Lie superalgebra $\mathfrak{osp}(1|2)$  (see \cite[Proposition 1]{DeBieGenestVinet} for the Clifford algebra case and Proposition \ref{pro:orstedA} for the general case).

Dunkl multiplicities $k_1,\ldots,k_n$ are often assumed to be positive in the literature, as in \cite{DeBieGenestVinet}, or at least non-negative. In our approach this is not possible, since the operators $D_\PP$ require that the sum of the $k_i$'s is always negative when $n>1$.
However, it is still possible to assume that the multiplicities belong to the regular set $K^{reg}$ for the reflection group $\zz_2^n$ (Corollary \ref{cor:Kreg} and Remark \ref{rem:Kreg}), so to avoid singularities of the Dunkl intertwining operators (see e.g.\ \cite[\S2.4]{Rosler}).  

The number of Dunkl-regular functions spaces on an $(n+1)$-dimensional hypercomplex subspace $M$ is equal to the number of partitions of the set $[n]$ (Theorem \ref{teo:classification1}). However, many of these spaces are equivalent (Theorem \ref{teo:classification2}). It turns out that the number of inequivalent spaces is equal to the number of partitions of the number $n$. For example, there is only one Dunkl-regular function space when $M=\cc$ (the holomorphic functions), three inequivalent function spaces on the quaternions, with $M=\hh$, and 15 inequivalent spaces on the octonions, with $M=\oo$ (see Table \eqref{table}).

The forthcoming paper \cite{PerottiFueterTrees} will show how elements of Dunkl theory can be applied to hypercomplex analysis. We will use the properties of the Dunkl Laplacian to provide a general Fueter Theorem, which is valid on any hypercomplex subspace of a real alternative *-algebra.

We describe in more detail the structure of the paper. Section \ref{sec:pre} is devoted to preliminaries. We recall some basic definitions
about real alternative *-algebras and hypercomplex subspaces, monogenic and slice functions,  slice-regularity. 
We introduce the three fundamental differential operators on a hypercomplex subspace $M$, i.e., the Cauchy-Riemann operator $\dM$, the global operator $\difbarM$ associated to slice-regularity and the spherical Dirac operator $\Gamma_M$. 
We use $\difbarM$ to characterize slice-regular polynomials
on $M$ (Proposition \ref{pro:sliceregularpoly}).
We also prove 
the basic relation between the there operators $\dM$, $\difbarM$ and $\Gamma_M$ (Theorem \ref{teo:difference}). 

Section \ref{sec:Dunkl-Dirac_operators} concerns Dunkl-Dirac operators and the characterizations of sliceness and slice-regularity. Subsection \ref{subsec:Dunkl-Dirac_operators} introduces Dunkl operators associated to the reflection group $\zz_2^n$ and the Dunkl-Dirac operator $\DD_\B$ associated to a hypercomplex basis $\B$ of $M$.  Following \cite{DeBieGenestVinet}, the (super)Casimir operator $\sC$ for the $\mathfrak{osp}(1|2)$ realization given by the operators $\x=\IM(x)$  and $\DD_\B$ is defined. Subsection \ref{sub:dirac_dunkl_operators_and_sliceness} gives necessary and sufficient conditions for sliceness related to the Casimir operator and to the spherical Dunkl-Dirac operator $\tGamma$. We begin with a criterion for polynomial functions (Theorem \ref{teo:poly_sliceness}) and we generalize it to any $C^1$ function by means of a supplementary pairs of Dunkl operators (Theorem \ref{teo:C1_sliceness}). Subsection \ref{sub:dirac_dunkl_monogenicity_and_slice_regularity} is devoted to characterizations of slice-regularity in the context of Dunkl theory (Theorems \ref{teo:poly_slice_regularity} and \ref{teo:slice_regularity}). 

Section \ref{sec:dunkl_regular_function_spaces} introduces Dunkl-regular function spaces. Subsection \ref{sub:Intermediate_Dunkl-Dirac_operators} defines, for any subset $A$ of $[n]$, the intermediate Dunkl-Dirac operator $D_A$, the `spherical' operator $\mathscr S_A$ and the function space $\F_{A}(\OO)$ of \emph{$A$-Dunkl-regular functions} on an open subset $\OO$ of $M$. 
Subsection \ref{sub:examples} classifies the $A$-Dunkl-Dirac functions spaces when the hypercomplex subspace $M$ is the whole alternative algebra of quaternions, the subspace of reduced quaternions, the paravector space of a Clifford algebra or the space of octonions. 
Subsection \ref{sub:Slice_Dunkl-regular_functions} consider the case of Dunkl-Dirac functions that are also slice functions on $M$. 
Subsection \ref{sub:Dunkl-regular_function_spaces_defined_by_partitions} introduces, for any open subset $\OO$ of $M$ and any partition $\PP$ of the set $[n]$, where $n=\dim M-1$, a space $\F_\PP(\OO)$ of functions in the kernel of a Dunkl-Cauchy-Riemann operator $D_\PP$, called \emph{$\PP$-Dunkl-regular functions}.  
Theorem \ref{teo:classification1} studies the dependence of the spaces
$\F_\PP(\OO)$ on $\PP$ and on the Dunkl multiplicities, while Theorem \ref{teo:classification2} studies equivalence of Dunkl-regular function spaces up to reordering of the hypercomplex basis and of the partition. 
Subsection \ref{sub:Dunkl-regularity_and_regular_$T$-functions} studies the relation between Dunkl-regularity and strongly $T$-regular functions, showing (Theorem \ref{teo:Tregular}) that for any $T$ there is a choice of $\PP$ such that $\F_\PP(\OO)$ is the set of strongly $T$-regular functions on $\OO$. In particular, any strongly $T$-regular function is Dunkl monogenic. We also introduce the concept of $\PP$-slice function on $M$ (Definition \ref{def:pslice}) and its characterization (Theorem \ref{teo:kerSP}) through Dunkl operators.

\section{Preliminaries}\label{sec:pre}

Let $\Aa$ be a finite-dimensional real alternative algebra with unity $1\ne0$. Alternativity means that the associator $[a,b,c]:=(ab)c-a(bc)$ is an alternating function of $a,b,c\in\Aa$.
We identify the real multiples of $1$ in $\Aa$ with the  real numbers. 
Assume that $\Aa$ is a *-algebra, i.e., it is equipped with a real linear anti-involution $x\mapsto x^c$ such that $(xy)^c=y^cx^c$ for all $x,y\in \Aa$ and  $x^c=x$ for $x$ real. Let $t(x):=x+x^c\in \Aa$ be the \emph{trace} of $x$ and $n(x):=xx^c\in \Aa$  the \emph{norm} of $x$. 
Let
\[
\s_{\Aa}:=\{J\in \Aa : t(x)=0,\ n(x)=1\}
\]
be the set of imaginary units of $\Aa$ compatible with the *-algebra structure of $\Aa$. Note that if $J\in\s_{\Aa}$, then $J^2=-JJ^c=-1$. Assuming $\s_{\Aa}\ne\emptyset$, one can define the \emph{quadratic cone} of $\Aa$ (see \cite[Definition 3]{AIM2011}) as the subset of $\Aa$
\[
\Q_{\Aa}:=\cup_{J\in \s_{\Aa}}\C_J,
\]
where $\C_J=\Span(1,J)$ is the complex `slice' of $\Aa$ generated by $1$ and $J$. It holds $\C_J\cap\C_K=\R$ for each $J,K\in\s_{\Aa}$ with $J\ne\pm K$. The quadratic cone is a real cone invariant w.r.t.\ translations along the real axis. 
Each element $x$ of $\Q_{\Aa}$ can be expressed as $x=\RE(x)+\IM(x)$, with $\RE(x)=\frac{x+x^c}2$, $\IM(x)=\frac{x-x^c}2=\beta J$, where $\beta=\sqrt{n(\IM(x))}\geq0$ and $J\in\s_{\Aa}$, with unique choice of $\beta\geq0$ and $J\in\s_{\Aa}$ if $x\not\in\R$.
In the following we will also write $\x$ in place of $\IM(x)$. 
Observe that the quadratic cone is properly contained in $\Aa$ unless $\Aa$ is isomorphic as a real *-algebra to one of the division algebras $\C,\HH,\mathbb O$ with the standard conjugations (see \cite[Proposition 1]{AIM2011}).
We refer to 
\cite[\S2]{AIM2011} and \cite[\S1]{AlgebraSliceFunctions} for more details and examples about real alternative *-algebras and their quadratic cones.

\subsection{Hypercomplex subspaces}\label{sec:Hypercomplex_subspaces}

We recall some concepts defined in \cite[\S3]{CRoperators}.
A real vector subspace $M$ of $\Aa$ with dimension $\dim(M)\ge2$ and such that $\R \subseteq M \subseteq \Q_{\Aa}$ 
is called a \emph{hypercomplex subspace of $\Aa$}. 

If $\Aa$ is the skew-field $\hh$ of quaternions or the space $\oo$ of octonions, then $\Aa$ itself is a hypercomplex subspace. The algebra $\hh$ contains also the hypercomplex subspace  $\hh_r=\{x=x_0+ix_1+jx_2\in\hh\;|\; x_0,x_1,x_2\in\R\}$ of \emph{reduced quaternions}. More generally,  every real subspace of the quaternions/octonions properly including the real axis is a hypercomplex subspace of the quaternions/octonions. 
The paravector space of the real Clifford algebra $\R_n$ of signature $(0,n)$, as well as \cite[Example 1.15]{AlgebraSliceFunctions} and \cite[Lemma 2.7]{GhiloniStoppato_arXiv24} are hypercomplex subspaces of $\R_n$.
Another example of an eight-dimensional hypercomplex subspace is provided by the subspace $M=\langle1,e_1,e_2,e_3,e_4,e_5,e_6,e_{123456}\rangle$ of $\R_6$.
Observe that a hypercomplex subspace $M$ coincides with the whole algebra $\Aa$ if and only if $M$ is $\cc$, $\hh$ or $\oo$ \cite[Proposition 1]{AIM2011}. 

\begin{lemma}\cite[Lemma 1.4]{VolumeCauchy} \label{lem:gis}
Let $M$ be a hypercomplex subspace of $\Aa$. Then there exists a norm $\|\ \|$ on $\Aa$ such that $\|x\|^2=n(x)$ for every $x \in M$. Let $(\ ,\,)$ be the scalar product  on $\Aa$ associated to the norm $\|\ \|$. Then it  holds $(x,y)=\frac12 t(xy^c)$ for every $x,y\in M$. 
\end{lemma}

Let $\B=(v_0,v_1,\ldots,v_{n})$ be a real vector basis of $M$ with $v_0=1$, orthonormal w.r.t.\ the scalar product $(\ ,\,)$. Complete $\B$ to a real vector basis $\B_{\Aa}=(v_0,v_1,\ldots,v_{d-1})$ of $\Aa$, orthonormal w.r.t.\ the scalar product $(\ ,\,)$ on $\Aa$.  
Then it holds $\{v_1,\ldots,v_n\}\subseteq \s_M:=\s_{\Aa}\cap M$ and $M$ is the orthogonal direct sum of $\R=\Span(v_0)$ and $M \cap \ker(t)=\Span(v_1,\ldots,v_n)$. 
Moreover, for every $i\ne j$ in the set $\{1,\ldots,n\}$, the elements $v_i$ and $v_j$ anticommute  (see \cite[\S3]{CRoperators}). Such a basis $\B$ is called, as in \cite{GhiloniStoppatoJGP}, a \emph{hypercomplex basis} of $M$. Lemma \ref{lem:gis} shows that every hypercomplex subspace has hypercomplex bases. Note also that it holds $\x^2=-n(x)=-\|x\|^2$ for every $x\in M$. 
From the equalities $0=[v_i,v_j,a]+[v_j,v_i,a]=(v_iv_j)a-v_i(v_ja)+(v_jv_i)a-v_j(v_ia)=-v_i(v_ja)-v_j(v_ia)$,
it follows the useful identity
\begin{equation}\label{eq:anticommute}
v_i(v_j a)=-v_j(v_ia)\text{\quad for every $a\in \Aa$ and $i, j\in\{1,\ldots,n\},\ i\ne j$.}  
\end{equation}

Let $L:\R^d \to \Aa$ be the real vector isomorphism defined by $\B_{\Aa}$, namely the map $x=(x_0,x_1,\ldots,x_{d-1})\mapsto L(x)=\sum_{\ell=0}^{d-1}x_{\ell}v_{\ell}$. Identify $\R^d$ with $\Aa$ via~$L$ and $M$ with $\R^{n+1}\simeq\R^{n+1} \times \{0\} \subseteq 
\R^d$. The product of $\Aa$ induces a product on $\R^d$: given $x,y \in \R^d$, $xy$ is defined as $L^{-1}(L(x)L(y))$. Since $\B_{\Aa}$ is orthonormal, $\|x\|$ coincides with the Euclidean norm $(\sum_{\ell=0}^{d-1}x_{\ell}^2)^{1/2}$ of $x$ in $\R^d$. Moreover,
\[
\s_M=\s_{\Aa}\cap M=\{L(x)\in M\,|\,x\in\R^{n+1},\,x_0=0,\, \textstyle\sum_{i=1}^nx_i^2=1\}
\]
is compact.


\subsection{Cauchy-Riemann operators on hypercomplex subspaces}\label{sec:CR_operators}

We recall from \cite[Def.\ 2]{CRoperators} the following definition. See also \cite[\S2.3]{GhiloniStoppato_arXiv24}. 
Let $\OO$ be an open subset of the hypercomplex subspace $M$ with hypercomplex basis $\B$ and coordinates $x_0,x_1,\ldots,x_n$. 
For $i=0,1,\ldots,n$ consider the differential operators $\partial_{x_i}:C^1(\OO,\Aa)\to C^0(\OO,\Aa)$ defined by 
\[
\textstyle\partial_{x_i} f=L\circ\dd {(L^{-1}\circ f\circ L_{|\OO'})}{x_i}\circ L^{-1},
\]
where $\OO'=L^{-1}(\OO)$. As noted in \cite[Remark 2.37]{GhiloniStoppato_arXiv24}, the operators $\partial_{x_i}$ depend only on the hypercomplex basis $\B$ and not on the choice of its extension $\B_{\Aa}$. Observe that $\partial_{x_0}$ does not even depend on $\B$ since $v_0=1$ for every hypercomplex basis. 

\begin{definition}\label{def:db}
Fix a hypercomplex basis  $\B=(1,v_1,\ldots,v_{n})$ of $M$. The \emph{Cauchy-Riemann operator induced by $\B$} is the differential operator $\dB:C^1(\OO,\Aa)\to C^0(\OO,\Aa)$ defined by
\[
\dB:=\partial_{x_0}+v_1\partial_{x_1}+\cdots +v_n\partial_{x_n}.
\]
\end{definition}

Sometimes in the literature a factor $1/2$ is present in front of the operator (as in \cite{CRoperators}). 

\begin{remark}
The operator $\dB$ does not depend on the choice of the hypercomplex basis $\B$ of $M$. This property generalizes a well-known fact about the Dirac operator on Clifford algebras and the Fueter-Moisil operator \cite{DentoniSce}  on the space $\oo$. 
If $\B'=(1,v'_1,\ldots,v'_{n})$ is another hypercomplex basis of $M$, then there exists a real orthogonal matrix $P\in O(n)$ such that $v_j=\sum_{i=1}^np_{ij}v'_i$ for every $j=1,\ldots,n$. If $L'$ is the isomorphism defined by $\B'_{\Aa}$, namely $L'(y_0,\ldots,y_{d-1})=\sum_{\ell=0}^{d-1}y_{\ell}v'_{\ell}$, then it holds $x_j=\sum_{i=1}^n p_{ij}y_i$ for $j=1,\ldots,n$. Therefore
\begin{align*}
\dB f&=\partial_{x_0} f+\textstyle\sum_{j=1}^nv_j\,\partial_{x_j} f=\partial_{x_0} f+\textstyle\sum_{i,j,h=1}^n p_{ij}v'_i p_{hj}\,\partial_{y_h} f=
\partial_{x_0} f+\textstyle\sum_{i=1}^n v'_i\,\partial_{y_i} f=\overline\partial_{\B'}f.
\end{align*}
\end{remark}

\begin{definition}\label{def:dM}
We define the \emph{Cauchy-Riemann operator of $M$} as $\dM:=\dB$ for any choice of the hypercomplex basis $\B$ of $M$. 
\end{definition}

\begin{examples}
The Cauchy-Riemann-Fueter operator on $\hh$, the Cauchy-Riemann (or Fueter-Moisil \cite{DentoniSce}) operator on the octonionic space $\oo$ and the Cauchy-Riemann operator defined on paravectors of the Clifford algebra $\R_n=\R_{0,n}$, are all examples of operators $\dM$ (with $M=\hh,\oo,\R^{n+1}$ respectively).
\end{examples}

Let $\partial_{\B}:=\partial_{x_0}-v_1\partial_{x_1}-\cdots -v_n\partial_{x_n}$ denote the conjugated Cauchy-Riemann operator induced by $\B$ and let $\Delta_\B$ be the Laplacian operator on $M$ induced by $\B$, acting on functions $f$ of class $C^2(\OO,\Aa)$ as
\[\Delta_\B f:=\textstyle\sum_{i=0}^n\partial_{x_i}(\partial_{x_i}f)=L\circ\Delta(L^{-1}\circ f\circ L_{|\OO'})\circ L^{-1},
\]
where $\OO'=L^{-1}(\OO)$ and $\Delta$ is the Laplacian of $\R^{n+1}$. 
Proposition 5 in \cite{CRoperators} proved that $\partial_\B(\dB f)=\dB(\partial_\B f)=\Delta_\B f$ for every function $f$ of class $C^2(\OO,\Aa)$. 
As observed above for $\dB$, also $\partial_{\B}$  and then $\Delta_\B$ do not depend on the choice of $\B$. We can then define the \emph{conjugated Cauchy-Riemann operator of $M$} and the \emph{Laplacian operator of $M$} as
\[
\partial_M:=\partial_{\B}\text{\quad  and\quad}\Delta_M:=\Delta_\B.
\]
It follows immediately that  $\Delta_Mf=\partial_M(\dM f)=\dM(\partial_Mf)$ for every $f\in C^2(\OO,\Aa)$.

\begin{definition}\label{def:monogenic}
Given an open set $\OO\subseteq M$, the functions $f\in C^1(\OO,\Aa)$ in the kernel of $\dM$ are called (left) \emph{monogenic functions} on $\OO$. We write $f\in\M(\OO)$. Observe that every monogenic function $f\in C^2(\OO,\Aa)$ is \emph{harmonic}, i.e., in the kernel of $\Delta_M$. 
\end{definition}

\begin{remark}
Let $\OO\subseteq M$ be open. 
A function $f\in C^2(\OO,\Aa)$ is harmonic if and only if, for any hypercomplex basis $\B$, the function $L^{-1}\circ f\circ L_{|\OO'}$ is harmonic w.r.t.\ the standard Laplacian of $\R^{n+1}$. In particular, every monogenic function $f\in C^2(\OO,\Aa)$ is real analytic. If a monogenic function $f$ is of class $C^1(\OO,\Aa)$, then $f$ is still harmonic and real analytic. This follows from the mean value property (see \cite[Proposition 3.16]{GhiloniStoppato_arXiv24} for $\Aa$ associative and \cite[Theorem 3.12]{Huo_Ren_Xu_arXiv25} for the general case) satisfied by $f$ (whence by $L^{-1}\circ f\circ L_{|\OO'}$). 
See also \cite[p.255]{DentoniSce} about the harmonicity of a monogenic $f$ of class $C^1$ in the octonionic case $\Aa=M=\oo$. 
\end{remark}

For more properties of monogenic functions we refer to \cite[\S3]{GhiloniStoppato_arXiv24} and \cite{Huo_Ren_Xu_arXiv25}, where it is shown that several classical  results of Clifford monogenic function theory (see, e.g., \cite{GHS}) can be extended to monogenic functions on hypercomplex subspaces of associative algebras. 

\subsection{Slice functions and slice-regular functions}\label{sec:slice}

\emph{Slice functions} on $\Aa$ are functions that are compatible with the decomposition of the quadratic cone into slices. Let the elements of $\Aa\otimes_{\R}\C$ be expressed as $w=a+\ui b$ with $a,b\in \Aa$ and $\ui^2=-1$ and consider the conjugation mapping $w=a+\ui b$ to $\overline w=a-\ui b$  for all $a,b\in \Aa$. Given a subset $D$ of $\C$ invariant w.r.t.\ complex conjugation, a function $F: D \to \Aa\otimes_{\R}\C$ is a \emph{stem function} if it satisfies  $F(\overline z)=\overline{F(z)}$ for every $z\in D$. For every $J\in\s_{\Aa}$, let $\phi_J:\C\to\C_J$ be the *-algebra isomorphism 
\[
\phi_J(\alpha+i\beta):=\alpha+J\beta\text{\quad for all $\alpha,\beta\in\R$}.
\]
Let $\OO_D$ be the \emph{axially symmetric} (or \emph{circular} \cite{AIM2011}) subset of the quadratic cone defined by 
\[
\OO_D=\bigcup_{J\in\s_{\Aa}}\phi_J(D)=\{\alpha+J\beta\in\Aa : \alpha,\beta\in\R, \alpha+i\beta\in D,J\in\s_{\Aa}\}.
\]
An axially symmetric connected set $\OO_D$ is called a \emph{symmetric slice domain} if $\OO_D\cap \R\ne\emptyset$, a \emph{product domain} if $\OO_D\cap \R=\emptyset$. 
The stem function $F=F_\emptyset+\ui F_1:D \to \Aa\otimes_{\R}\C$  induces the \emph{(left) slice function} $f=\I(F):\OO_D \to \Aa$: if $x=\alpha+J\beta =\phi_J(z)\in \OO_D\cap \C_J$, then  
\[ f(x)=F_\emptyset(z)+JF_1(z),\text{\quad where $z=\alpha+i\beta\in D$}. 
\]

Suppose that $D$ is open.  The slice function $f=\I(F):\OO_D \to \Aa$ is called \emph{(left) slice-regular} if $F$ is holomorphic w.r.t.\ the complex structure on $\Aa\otimes_{\R}\C$ defined by left multiplication by $\ui$. 
We will denote by $\SL^1(\OO_D)$ the real vector space of slice functions induced by stem functions of class $\mathcal C^1$ on $\OO_D$ and by $\sr(\OO_D)$ the vector subspace of slice-regular functions on $\OO_D$. 
For example, polynomial functions $f(x)=\sum_{j=0}^d x^ja_j$ and convergent power series with right coefficients in $\Aa$ are slice-regular, while polynomials in $x,x^c$ of the form 
$\sum_{\alpha+\beta=k}x^\alpha (x^c)^\beta a_{\alpha,\beta}$, with coefficients $a_{\alpha,\beta}\in\Aa$, are slice functions.
If $\Aa=\hh$ and $\OO_D$ is a symmetric slice domain, this definition of slice-regularity is equivalent to the original one proposed by Gentili and Struppa in \cite{GeSt2007Adv}. 

If $M\subseteq\Q_{\Aa}$ is a hypercomplex subspace, we can consider the restriction of a slice function defined on $\OO_D$ to the subset $\OO=\OO_D\cap M$ of $M$. We call such a set $\OO$ an \emph{axially symmetric} open subset of $M$. 
Thanks to the representation formula (see e.g.~\cite[Proposition 6]{AIM2011}), the restriction of a slice function $f=\I(F)\in\SL(\OO_D)$ to $\OO$ uniquely determines $f$ and the inducing stem function $F$. We will use the same symbol $f$ to denote the restriction of $f$ to $\Omega$, and the symbols $\SL(\OO)$ and $\sr(\Omega)$ to denote respectively the set of (restricted) slice functions and the set of (restricted) slice-regular functions on $\OO$, where $\OO$ is an axially symmetric open subset of $M$.

To any function $f:\OO_D \to \Aa$, not necessarily slice, one can associate the function $\vs f:\OO_D \to \Aa$, called \emph{spherical value} of $f$, and the function $f'_s:\OO_D \setminus \R \to \Aa$, called  \emph{spherical derivative} of $f$, defined as
\begin{equation}\label{eq:sphericalfunctions}
\vs f(x):=\tfrac{1}{2}(f(x)+f(x^c))
\quad \text{and} \quad
f'_s(x):=\tfrac{1}{2}\,\x^{-1}(f(x)-f(x^c)),
\end{equation}
where $\x=\IM(x)$. The following equalities hold
\begin{equation}\label{eq:spherical}
f=\vs f+\x f'_s=\vs f+\x ^{-1}\vs{(\x f)}.
\end{equation}
If $f:\OO_D\to\Aa$ is a slice function, 
then the functions $\vs f$ and $\sd f$ are constant on every set $\s_x:=\alpha+\beta\,\s_{\Aa}$, $x=\alpha+I\beta\in\OO_D\setminus\rr$. Since also $\x f$ is slice, also 
the function $\vs{(\x f)}$ is constant on the sets $\s_x$, with $x\in\OO_D\setminus\rr$. Moreover, 
\begin{equation}\label{eq:F01}
F_\emptyset(z)=\vs f(x),\quad F_1(z)=\beta\sd f(x)=-\beta^{-1}\vs{(\x f)}(x).  
\end{equation}

\begin{remark}\label{rem:sliceness_continuity_M}
If $M\subseteq\Q_{\Aa}$ is a hypercomplex subspace and $\OO=\OO_D\cap M$ is axially symmetric, then 
for any functions $f\in\SL(\OO)$ the functions $\vs f$, $\sd f$ and $\vs{(\x f)}$ defined as in \eqref{eq:sphericalfunctions} are constant on every sphere $\s_x\cap M=\alpha+\beta\,\s_M$, with $x=\alpha+I\beta\in\OO\setminus\rr$ and $\s_M=\s_{\Aa}\cap M$. 
\end{remark}

\begin{proposition}\label{pro:sliceness_continuity_M}
If $f$ is continuous on $\OO\subseteq M$, then $f$ is a slice function on $\OO$ if and only if it is a slice function on $\OO\setminus\rr$. 
Therefore $f\in\SL(\OO)\cap C(\OO,\Aa)$ if and only if the functions $\vs f$ and $\vs{(\x f)}$ (or $\vs{(xf)}$) are constant on every sphere $\s_x\cap M=\alpha+\beta\,\s_M$, with $x=\alpha+I\beta\in\OO\setminus\rr$. 
\end{proposition}
\begin{proof}
We can assume that $\OO\cap\rr\not=\emptyset$. 
If $x=\alpha+I\beta\in\OO=\OO_D\cap M$ and $f=\I(F_\emptyset+\ui F_1)$ on $\OO\setminus\rr$, then
\[
\|F_1(z)\|=\|-I(f(x)-F_\emptyset(z))\|\le\omega\|f(x)-F_\emptyset(z)\|\le\omega(\|f(x)-f(\alpha)\|+\|f(\alpha)-F_\emptyset(z)\|)\to0
\]
as $z$ tends to $\alpha\in D\cap\R$, where $\omega$ is a positive constant such that $\|xa\|\le\omega\|x\|\|a\|$ for every $x\in M$, $a\in\Aa$, as in \cite[Remark\ 2.27]{GhiloniStoppato_arXiv24}. Setting $F_\emptyset(\alpha):=f(\alpha)$ and $F_1(\alpha):=0$, one obtains a continuous stem function on $D$ which induces $f$ on $\OO$.
The last statement follows from \eqref{eq:F01}: if $\vs f$ and $\vs{(\x f)}$ are constant on every sphere $\s_x$, with $x\in\OO\setminus\rr$, then setting
\[
F_\emptyset(z):=\vs f(x),\quad F_1(z):=-\beta^{-1}\vs{(\x f)}(x)
\]
we get a stem function on $\OO\setminus\rr$ that induces $f$, whence $f\in\SL(\OO)\cap C(\OO,\Aa)$.  Finally, observe that $\vs{(xf)}=x_0\vs f+\vs{(\x f)}$.
\end{proof}

Observe that when the hypercomplex subspace has dimension 2, then any function $f:\OO\to\Aa$ is a slice function, since in this case $\alpha+\beta \s_M=\{x,x^c\}$ and $\vs f$, $\vs{(\x f)}$ are invariant w.r.t.\ conjugation $x\mapsto x^c$. 

The \emph{slice derivatives} $\dd{f}{x},\dd{f\;}{x^c}$ of a $C^1$ slice function $f=\I(F)$ are defined by means of the  Cauchy-Riemann operators applied to the inducing stem function $F$:
\[\dd{f}{x}=\I\left(\dd{F}{z}\right),\quad \dd{f\;}{x^c}=\I\left(\dd{F}{\overline z}\right).\]
It follows that $f$ is slice-regular if and only if $\dd{f\;}{x^c}=0$ and if $f$ is slice-regular on $\OO$ then $\dd{f}{x}$ is slice-regular as well. Moreover, the slice derivatives satisfy the Leibniz product rule w.r.t.\ the slice product. 
We refer the reader to \cite[\S3.4]{AIM2011} for more properties of slice functions and slice-regularity.

\subsection{The global operators $\difM$, $\difbarM$ and the spherical Dirac operator}

Let $\OO$ be an open subset of the hypercomplex subspace $M$ of $\Aa$. 
We recall from \cite{Gh_Pe_GlobDiff} and \cite{CRoperators} the definition of the global differential operators $\dif,\difbar: C^1(\OO\setminus \R,A) \to C^0(\OO\setminus \R, A)$ associated with the slice derivatives:  
\begin{equation}\label{eq:theta}
\dif=\partial_{x_0}+\x ^{-1}\Eu\text{\quad and\quad} \difbar=\partial_{x_0}-\x ^{-1}\Eu,
\end{equation}
where $\Eu=\sum_{i=1}^nx_i\partial_{x_i}$ is the Euler operator for $M\cap\ker(t)=\Span(v_1,\ldots,v_n)$. 
The operator $\Eu$ does not depend on the choice of the hypercomplex basis $\B$ whence also the operators $\dif,\difbar$ do not depend on  $\B$.

\begin{definition}\label{def:thetaM}
We define the operators $\difM$ and $\difbarM$ as $\difM:=\dif$, $\difbarM:=\difbar$ for any choice of the hypercomplex basis $\B$ of $M$. 
\end{definition}

Observe that for every slice function $f$ on $\OO$, it holds $\frac12\difM f=\dd{f}{x}$ and $\frac12\difbarM f =\dd f{x^c}$  on $\OO\setminus\R$ (see \cite[Theorem 2.2]{Gh_Pe_GlobDiff}). 
Therefore if $f\in\SL(\OO)$, $f\in\sr(\OO)$ if and only if $\difbarM f=0$ on $\OO\setminus\R$. 

The operator $\difbarM$ provides the general form of slice-regular polynomials with right coefficients in $\Aa$. 

\begin{proposition}\label{pro:sliceregularpoly}
Let $\B=(1,v_1,\ldots,v_n)$ be a hypercomplex basis of $M$, with coordinates $x_0,\ldots,x_n$. 
Let $f\in \Aa[x_0,\ldots,x_n]$ be a polynomial function with coefficients in $\Aa$. If $f$ is slice-regular, then $f$  has the form $f(x)=\sum_{j=0}^d x^ja_j$ for some $a_j\in\Aa$, where $x=x_0+\sum_{i=1}^nx_iv_i$.
\end{proposition}
\begin{proof}
Firstly, decompose $f$ in its homogeneous components $P_i$ of degree $i$
\[
f=P_d+P_{d-1}+\cdots+P_1+P_0.
\]
 Since $f$ is slice-regular, it holds 
\[
0=\difbar f=\partial_{x_0}f-\x^{-1}\Eu f=\sum_{i=0}^d\partial_{x_0}P_i-\x^{-1}\sum_{i=0}^d\Eu P_i 
\]
where $\x=\IM(x)$. Equivalently, $\sum_{i=0}^d\left(\x\partial_{x_0}P_i-\Eu P_i\right)=0$. Therefore every homogeneous component $\x\partial_{x_0}P_i-\Eu P_i$ must be zero, that is $\difbar P_i=0$ for $i=0,\ldots,d$.
It is then sufficient to prove that if the real components of $f$ are homogeneous polynomials of degree $k$ in $x_0,\ldots,x_n$, then there exists $a\in\Aa$ such that $f(x)=x^k a$. Since $f$ is slice-regular, it holds 
\[
0=\difbar f=\partial_{x_0}f-\x^{-1}((x_0\partial_{x_0}+\Eu )f-x_0\partial_{x_0}f)
=\partial_{x_0}f-\x^{-1}(k f-x_0\partial_{x_0}f).
\]
This implies that $kf(x)=(x_0+\x)\partial_{x_0}f=x\, \partial_{x_0}f$. Since $\partial_{x_0}$ and $\difbar$ commute, we can repeat the computation with $\partial_{x_0}f\in\ker \difbar$ and get, inductively, that
\[
k! f(x)=x^k\partial^k_{x_0}f(x).
\]
This means that $f(x)=x^k a$ for some $a\in\Aa$.
\end{proof}

Now we introduce the \emph{spherical Dirac operator} for $M$. 
For any $i,j$ with $1\le i,j\le n$, let $L_{ij}=x_i\partial_{x_j}-x_j\partial_{x_i}$ and let
\[
\GB=-\sum_{1\le i<j\le n} v_i(v_jL_{ij})=-\tfrac12\sum_{i,j=1}^n v_i(v_jL_{ij})
\] 
be the \emph{spherical Dirac operator} associated to the basis $\B$. 
It acts on $C^1$ functions $f$ as $\GB f=-\tfrac12\sum_{i,j=1}^n v_i(v_jL_{ij}f)$. The two expressions of $\GB$ given above are equivalent thanks to \eqref{eq:anticommute}. 
The operators $L_{ij}$ are tangential differential operators for the spheres $\s_x\cap M=\alpha+\beta\s_M$, with $x=\alpha+I\beta\in\Q_{\Aa}\setminus\rr$. 
We now prove the main relation linking the operators $\dM$, $\difbarM$ and $\Gamma_\B$. This result improves what obtained in \cite[Theorem 8]{CRoperators}.  

\begin{theorem}\label{teo:difference}
Let $f:\OO\to \Aa$ be a $\mathcal{C}^1$ function. 
Then the following formula holds in $\OO\setminus\R$:
\[
\dM f-\difbarM f=-\x ^{-1} \Gamma_\B f.
\]
\end{theorem}
\begin{proof}
Using \eqref{eq:anticommute}, we obtain
\begin{align*}
&\x(\Eu+\Gamma_\B)f=
\x\left(\sum_{i=1}^n x_i\partial_{x_i}f+\Gamma_\B f\right)=\sum_{k=1}^nx_kv_k\sum_{i=1}^n x_i\partial_{x_i}f-\tfrac12\sum_{k=1}^nx_kv_k\sum_{i,j}v_i(v_j(L_{ij}f))\\
&=\sum_{i,k}x_kx_iv_k\dds if-\tfrac12 \sum_{i,j,k}x_kx_i v_k(v_i(v_j\dds jf))+\tfrac12 \sum_{i,j,k}x_kx_j v_k(v_i(v_j\dds if))\\
&=\sum_{i,k}x_kx_iv_k\dds if+\tfrac12 \sum_{k,j}x_k^2 v_j\dds jf+ \tfrac12\sum_{\stackrel{i,j,k}{i\ne j}}x_kx_j v_k(v_i(v_j\dds if))-\tfrac12 \sum_{i,k}x_kx_i v_k\dds if\\
&=\tfrac12\sum_{i,k}x_kx_iv_k\dds if+\tfrac12 n(\x)\sum_j v_j\dds jf+ \tfrac12\sum_{\stackrel{i,j,k}{i\ne j,j\ne k}}x_kx_j v_k(v_i(v_j\dds if))+ \tfrac12\sum_{\stackrel{i,k}{i\ne k}}x_k^2 v_k(v_i(v_k\dds if))\\
&=\tfrac12\sum_{i,k}x_kx_iv_k\dds if+\tfrac12 n(\x)\sum_j v_j\dds jf+ 
\tfrac12\sum_{i\ne j} x_ix_j v_i(v_i(v_j\dds if))+ \tfrac12\sum_{\stackrel{i,k}{i\ne k}}x_k^2 v_i\dds if\\
&=\tfrac12\sum_{i,k}x_kx_iv_k\dds if+n(\x)\sum_j v_j\dds jf
-\tfrac12\sum_{i\ne j} x_ix_j v_j\dds if-\tfrac12 \sum_{i}x_i^2 v_i\dds if\\
&=-\x^2\sum_j v_j\dds jf
,
\end{align*}
where we used the fact that the sum of $x_kx_j v_k(v_i(v_j\dds if))$ over all distinct indices $i,j,k$ in the set $\{1,\ldots,n\}$ vanishes, since it is antisymmetric w.r.t.\ the indices $j$ and $k$. 
Therefore we have
\[
\x^{-1}\left(\sum_{i=1}^m x_i\partial_{x_i}f+\Gamma_\B f\right)=-\sum_{j=1}^m v_j\dds jf\text{\quad on $\OO\setminus\rr$}
\]
from which we deduce that $
\dM f-\difbarM f=\sum_{i=1}^m v_i\dds if+\x^{-1}\sum_{i=1}^m x_i\partial_{x_i}f=-\x^{-1}\Gamma_\B f$.
\end{proof}

Theorem \ref{teo:difference} shows that also the spherical Dirac operator 
does not depend on the choice of $\B$. 

\begin{definition}\label{def:GammaM}
The \emph{spherical Dirac operator for $M$} is defined as $\Gamma_M:=\Gamma_\B$ for any choice of the hypercomplex basis $\B$ of $M$. 
\end{definition}
In view of the foregoing definition, we can rewrite the formula in Theorem \ref{teo:difference} as 
\begin{equation}\label{eq:difference}
\dM f-\difbarM f=-\x ^{-1} \Gamma_M f.
\end{equation}

\section{Dunkl-Dirac operators, sliceness and regularity}\label{sec:Dunkl-Dirac_operators}

\subsection{Dunkl-Dirac operators}\label{subsec:Dunkl-Dirac_operators}

We recall from \cite{DeBieGenestVinet} and \cite{Dunkl} the definition of the Dunkl operators associated to the Abelian reflection group $\zz_2^n=\zz_2\times\cdots\times\zz_2$. If $r_i$ is the reflection operator w.r.t.\ the $i$-th coordinate of $\rr^n$, i.e., for any function $f$,
\[
(r_if)(x)=f(x_1,\ldots,-x_i,\ldots,x_n)
\]
and $k_1,\ldots,k_n\in\rr$ are the multiplicities, the \emph{Dunkl operators} $T_1,\ldots,T_n$ on $\rr^n$ associated to the reflection group $\zz_2^n$ are defined as follows:
\[
T_i=\dd{}{x_i}+\frac{k_i}{x_i}(1-r_i),\quad i=1,\ldots,n.
\]
Here $1$ denotes the identity operator. The operators $T_i$ commute with each other and define the \emph{Dunkl-Laplace} operator $\Delta_D=\sum_{i=1}^nT_i^2$. 

We recall from \cite{Ren_et_al} the definition of the Dunkl-Dirac operator $\DD=\sum_{i=1}^ne_iT_i$ on the Clifford algebra $\R_n$ with generators $\{e_0=1,e_1,\ldots,e_n\}$.  Observe that $\DD$ coincides with the Dirac operator of $\R_n$ when the multiplicities $k_i$ are all zero. Now we extend this definition from Clifford algebras to any hypercomplex subspace.

\begin{definition}\label{def:DiracDunkl}
Let $\B=(1,v_1,\ldots,v_n)$ be a hypercomplex basis of the hypercomplex subspace $M$ of $\Aa$, with associated coordinates $x_0,x_1,\ldots,x_n$. Let $\OO\subseteq M$ be open, invariant w.r.t.\ reflections $r_i$, $i=1,\ldots,n$. We define the \emph{$\zz_2^n$ Dunkl-Dirac operator $\DD_\B:C^1(\OO,\Aa)\to C^0(\OO,\Aa)$} w.r.t.\ $\B$ 
as 
\[
\DD_\B:=\sum_{i=1}^nv_i T_{\B,i},\text{\quad where\quad} T_{\B,i}f=L\circ T_i {(L^{-1}\circ f\circ L_{|\OO'})}\circ L^{-1}=\partial_{x_i}f+\frac{k_i}{x_i}(f-r_i f),
\]
and $\OO'=L^{-1}(\OO)$. Here the operator $T_{\B,i}$ acts on every real component of an $\Aa$-valued function $f$. 
We also define the \emph{Dunkl-Cauchy-Riemann operator} of $M$ w.r.t.\ $\B$ as $D_\B:=\partial_{x_0}+\DD_\B$.
\end{definition}

The operators $T_{\B,i}$ commute with each other. Moreover, if $f\in C^\ell(\OO,\Aa)$, then  $T_i {(L^{-1}\circ f\circ L_{|\OO'})}\in C^{\ell-1}(\OO',\rr^d)$ whence $T_{\B,i}f\in C^{\ell-1}(\OO,\Aa)$ for any $\ell\ge1$ (see \cite[Lemma\ 2.9]{Rosler}). It follows that $D_\B f,\DD_\B f\in C^{\ell-1}(\OO,\Aa)$. Moreover, when ${\bf k}=(k_1,\ldots,k_n)=0$, then $D_\B=\dM$. In general, the Dunkl-Cauchy-Riemann operator can be expressed as
\begin{equation}\label{eq:db}
D_\B f=\dM f+\sum_{i=1}^n \frac{k_iv_i}{x_i}(f-r_i f).  
\end{equation}
If the segment in $M$ joining $x$ and the reflected point $r_i(x)$ is contained in $\OO$ for every $i=1,\ldots,n$ and $x\in\OO$, we can also write, as in \cite[Lemma\ 2.9]{Rosler}, 
\begin{equation}\label{eq:db_bis}
D_\B f=\dM f+2\sum_{i=1}^n k_iv_i\int_0^1\partial_{x_i}f(x-2tx_iv_i)dt.  
\end{equation}

\begin{definition}\label{def:Dunklmonogenic}
Given an open set $\OO\subseteq M$, invariant w.r.t.\ reflections $r_i$, the functions $f\in C^1(\OO,\Aa)$ in the kernel of $D_\B$ are called (left) \emph{Dunkl monogenic functions} on $\OO$. 
\end{definition}

By direct computation, remembering \eqref{eq:anticommute} in the nonassociative case, one obtains
\[
\DD_\B^2 =-\sum_{i=1}^n T^2_{\B,i}=:-\Delta_{D,n},\quad D_\B(\partial_{x_0}-\DD_\B)=(\partial_{x_0}-\DD_\B)D_\B =\partial_{x_0}^2+\Delta_{D,n}=:\Delta_{D,M}
\]
where $\Delta_{D,n}$ is the \emph{Dunkl-Laplace operator} of $M\cap\ker(t)$ and $\Delta_{D,M}$ is the  \emph{Dunkl-Laplace operator} of $M$. We recall that $t$ denotes the trace function on $\Aa$ (\S\ref{sec:pre}). 
Therefore every Dunkl monogenic function of class $C^2$ belongs to the kernel of $\Delta_{D,M}$, namely, it is \emph{Dunkl harmonic}.


For $x\in M$, let $\x =\IM(x)=\sum_{i=1}^n x_iv_i$. 
When $M$ is the paravector subspace of the Clifford algebra $\R_n$, the Dunkl-Dirac operator together with the left multiplication operator by $\x$ provide a realization of the orthosymplectic Lie superalgebra $\mathfrak{osp}(1|2)$ (see \cite[Lemma 4.1]{Orsted} and \cite[Theorem 1]{deBieOrstedSombergSoucek}). 

Let  $[\ ,\ ]$ denote the commutator and $\{\ ,\ \}$ the anticommutator of a pair of real-linear operators w.r.t.\ composition. 

\begin{proposition}\label{pro:orsted}
It holds 
\begin{align*}
&\left\{\x,\x\right\}=-2\|\x\|^2=-2n(\x);& &\left\{\DD_\B,\DD_\B\right\}=-2\Delta_{D,n}; && \left[\Eu,\x\right]=\x;\\
&\left\{\x,\DD_\B\right\}=-2(\Eu+\gamma);& &\left[\DD_\B,\Eu+\gamma\right]=\left[\DD_\B,\Eu\right]=\DD_\B; && \left[\Delta_{D,n},\x\right]=2\DD_\B;\\
& \left[\DD_\B,\|\x\|^2\right]=2\x,&&
\end{align*}
where $\gamma=\frac n2+\kappa$, $\kappa=\sum_{i=1}^nk_i$.
\end{proposition}
\begin{proof}
We can repeat the arguments of \cite[Lemma 4.1]{Orsted}, using \eqref{eq:anticommute} in the nonassociative case. The first three identities are trivial. We prove the fourth. It holds
\begin{align*}
&\x\,\DD_\B f+\DD_\B(\x f)=\sum_{i=1}^nx_iv_i\left(\sum_{j=1}^nv_j T_{\B,j}f\right)+
\sum_{i=1}^n v_i T_{\B,i}\left(\sum_{j=1}^n x_jv_jf\right)\\
&\quad=\sum_{i=1}^n\left(x_iv_i\left(v_i T_{\B,i}f\right)+v_i T_{\B,i}\left(x_iv_if\right)\right)+
\sum_{i\ne j}v_i\left(x_iv_j T_{\B,j}f+T_{\B,i}\left(x_jv_jf\right)\right)=:S_1+S_2.
\end{align*}
Using $\eqref{eq:anticommute}$ and the equality $T_{\B,i}(x_jf)=x_jT_{\B,i}\left(f\right)$, valid for $i\ne j$, we get
\begin{align*}
&S_2=\sum_{i\ne j}v_i\left(v_j\left(x_iT_{\B,j}f+T_{\B,i}\left(x_jf\right)\right)\right)=\sum_{i\ne j}v_i\left(v_j\left(x_iT_{\B,j}f+x_jT_{\B,i}\left(f\right)\right)\right)\\
&\quad+\sum_{i\ne j}v_i\left(v_j\left(T_{\B,i}(x_jf)-x_jT_{\B,i}\left(f\right)\right)\right)=0.
\end{align*}
Using the equality $v_i(v_i a)=-a$, valid for any $a\in\Aa$, we obtain
\begin{align*}
S_1&=-\sum_{i=1}^n\left(x_iT_{\B,i}f+T_{\B,i}\left(x_if\right)\right)\\
&=-\sum_{i=1}^n\left(x_i\partial_{x_i}f+\partial_{x_i}\left(x_if\right)\right)-\sum_{i=1}^n\left(k_i(f(x)-(r_if)(x))+\frac{k_i}{x_i}\left(x_if(x)+x_i(r_if)(x)\right)\right)
\\
&=-2\Eu f-nf-2\kappa f=-2\Eu f-2\gamma f
\end{align*}
and the equality $\left\{\x,\DD_\B\right\}=-2(\Eu+\gamma)$ is proved.
To prove the fifth equality, we firstly observe that $T_{\B,i}$ and $\partial_{x_j}$ commute for any $i\ne j$, and that
\begin{align*}
&T_{\B,i}(x_i\partial_{x_i}f)-x_i\partial_{x_i}(T_{\B,i}f)=
\partial_{x_i}f+x_i\partial^2_{x_i}f+\frac{k_i}{x_i}\left(x_i\partial_{x_i}f-r_i(x_i\partial_{x_i}f)\right)\\
&\quad-x_i\partial^2_{x_i}f-x_i\partial_{x_i}\left(\frac{k_i}{x_i}(1-r_i)f\right)
\\
&\quad =\partial_{x_i}f+k_i\partial_{x_i}f+\frac{k_i}{x_i}x_ir_i\partial_{x_i}f-k_i\partial_{x_i}\left((1-r_i)f\right)+\frac{k_i}{x_i}(1-r_i)f
\\
&\quad=\partial_{x_i}f+\frac{k_i}{x_i}(1-r_i)f
=T_{\B,i}f.
\end{align*}
In the third equality we used the anti-commutation property $\{r_i,\partial_{x_i}\}=0$. 
Therefore 
\begin{align*}
\left[\DD_\B,\Eu+\gamma\right]f=
\left[\DD_\B,\Eu\right]f
=\sum_{i=1}^nv_i\left(T_{\B,i}\left(x_i\partial_{x_i}f\right)-x_i\partial_{x_i}\left(T_{\B,i}f\right)\right)
=\DD_\B f.
\end{align*}
To prove that $[\Delta_{D,n},\x]=2\DD_\B$, we compute
\begin{align*}
&\DD_\B\DD_\B(\x f)-\x\DD_\B^2 f=\DD_\B(-\x\DD_\B f+\{\DD_\B,\x\}f)-\x\DD_\B^2f\\
&\quad=-\DD_\B(\x\DD_\B f)-2\DD_\B(\Eu+\gamma)f-\x\DD_\B^2f=-\{\DD_\B,\x\}\DD_\B f-2\DD_\B(\Eu+\gamma)f\\
&\quad=2(\Eu+\gamma)\DD_\B f-2\DD_\B(\Eu+\gamma)f=-2\left[\DD_\B,\Eu+\gamma\right]f=-2\DD_\B f.
\end{align*}
We conclude proving the seventh equality:
\begin{align*}
\left[\DD_\B,\|\x\|^2\right]f&=-\DD_\B(\x^2 f)+\x^2 \DD_\B f=\x\DD_\B (\x f)-\left\{\x,\DD_\B\right\}(\x f)+\x^2\DD_\B f=\\
&=\x\left\{\x,\DD_\B\right\}f+2(\Eu+\gamma)(\x f)=2\left[\Eu+\gamma,\x\right]f=2\left[\Eu,\x\right]f=2\x f.
\end{align*}

\end{proof}

It can be proved, following the Clifford algebra case \cite{deBieOrstedSombergSoucek}, that the commutation relations of Proposition \ref{pro:orsted} can be completed to show that we always realize the Lie superalgebra $\mathfrak{osp}(1|2)$. In the rest of the paper we will use only the relations proved in Proposition \ref{pro:orsted}. 

\begin{proposition}\label{pro:Lie}
On any hypercomplex subspace $M$ with hypercomplex basis $\B$, the operators $\x$  and $\DD_\B$ generate the Lie superalgebra $\mathfrak{osp}(1|2)$.\hfill\qed
\end{proposition}

Following \cite{DeBieGenestVinet}, we now define the \emph{(super)Casimir operator} for the $\mathfrak{osp}(1|2)$ realization of Proposition \ref{pro:Lie}:
\begin{equation}
\sC=\tfrac12\left(\left[\x,\DD_\B\right]-1\right)
\end{equation}
and the \emph{spherical Dunkl-Dirac operator} of $M$ w.r.t.\ $\B$:
\begin{equation}
\tGamma=\sC r,  
\end{equation}
where $r=\prod_{i=1}^nr_i$ is the composition of reflections $r_i$. These operators act on functions $f\in C^1(\OO,\Aa)$ as follows:
\[
\sC f=\tfrac12\left(\x\,\DD_\B f-\DD_\B(\x\, f)-f\right)\quad\text{and}\quad \tGamma f=\sC (f(x^c)).
\]

We generalize some properties proved in \cite{DeBieGenestVinet} over Clifford algebras.

\begin{proposition}\label{pro:operators}
\textrm{(i)}\ The operators $\sC$ and $\tGamma$ satisfy the following commutation relations:
\begin{equation}\label{eq:commutation}
  \left\{\sC,\DD_\B\right\}=0,\ \left\{\sC,\x\right\}=0,\ \left[\tGamma,\DD_\B\right]=0,\ \left[\tGamma,\x\right]=0.
\end{equation}
Moreover, it holds
\begin{equation}\label{eq:r}
  \{\DD_\B,r\}=\{\sC,r\}=\{\x,r\}=[\Eu,r]=0.
\end{equation}
\textrm{(ii)}\ If $\kappa=\sum_{i=1}^nk_i=(1-n)/2$, then the operator $\sC$ can be expressed as
\begin{equation}\label{casimir}
\sC=\x\,\DD_\B+\Eu.  
\end{equation}
\end{proposition}
\begin{proof}
(i)\quad We have
\begin{align*}
2\left\{\sC,\DD_\B\right\}f&=(\x\DD_\B-\DD_\B\x-1)\DD_\B f+\DD_\B(\x\DD_\B-\DD_\B\x-1)f\\
&=\x\DD_\B^2 f-\DD_\B f-\DD_\B^2(\x f)-\DD_\B f=-\x\Delta_{D,n}f-2\DD_\B f+\Delta_{D,n}(\x f)\\
&=\left[\Delta_{D,n},\x\right]f-2\DD_\B f=0,
\end{align*}
thanks to Proposition \ref{pro:orsted}, 
and
\begin{align*}
2\left\{\sC,\x\right\}f&=(\x\DD_\B-\DD_\B\x-1)(\x f)+\x(\x\DD_\B-\DD_\B\x-1)f\\
&=\x\DD_\B(\x f)-\DD_\B(\x^2 f)-\x f+\x^2\DD_\B f-\x\DD_\B(\x f)-\x f\\
&=\left[\x^2,\DD_\B\right]f-2\x f=-\left[\|\x\|^2,\DD_\B\right]f-2\x f=0.
\end{align*}
The relations \eqref{eq:r} are trivial. Using these and the proved equalities for $\sC$, we deduce the two identities involving $\tGamma$.

(ii)\quad If $\sum_{i=1}^nk_i=(1-n)/2$, then $\gamma=1/2$ and
\[
\sC=\frac12\left(\x\,\DD_\B-\DD_\B\x-1\right)=\frac12\left(2\x\,\DD_\B-\left\{\x,\DD_\B\right\}-1\right)=\x\,\DD_\B+\Eu+\gamma-1/2=\x\,\DD_\B+\Eu.
\]
\end{proof}

\subsection{Dunkl-Dirac operators and sliceness} 
\label{sub:dirac_dunkl_operators_and_sliceness}


We begin giving a necessary condition for sliceness related to the Casimir operator $\sC$. 

\begin{remark}\label{rem:circular}
Every \emph{circular} slice function, i.e., a slice function on a set $\OO$ which is constant on the spheres $\s_x$ for every $x\in\OO$, is in the kernel of the operator $\Gamma_M$. 
\end{remark}

\begin{proposition}\label{pro:kernels}
Assume $\sum_{i=1}^nk_i=(1-n)/2$. Let $\OO\subseteq M$ be open and axially symmetric and let $f\in\SL^1(\OO)$ be a slice function of class $C^1$. Then $\sC f=\tGamma f=0$. 
\end{proposition}
\begin{proof}
From the decomposition $f=\vs f+\x\, f'_s$ and \eqref{eq:commutation} we get $\sC f=\sC (\vs f)-\x\,\sC (f'_s)$. Using \eqref{casimir}, 
we get
\begin{align*}
\sC(\vs f)&=\x\,\DD_\B(\vs f)+\Eu(\vs f)=\sum_{i=1}^nx_iv_i\left(\sum_{j=1}^nv_j T_{\B,j}(\vs f)\right)+\sum_{i=1}^nx_i\partial_{x_i}(\vs f)\\
&=-\sum_{i=1}^nx_i(T_{\B,i}(\vs f)-\partial_{x_i}(\vs f))+\sum_{i<j}(x_i v_i(v_j T_{\B,j}(\vs f))+x_jv_j(v_iT_{\B,i}(\vs f))).
\end{align*}
Since $\vs f$ is constant on the spheres $\s_x\subset\OO$,  $T_{\B,i}(\vs f)=\partial_{x_i}(\vs f)$. It follows that
\begin{align*}
\sC(\vs f)&=\sum_{i< j}(v_i(v_j x_i\partial_{x_j}(\vs f))+v_j(v_ix_j\partial_{x_i}(\vs f)))=
\sum_{i<j}v_i(v_j(x_i \partial_{x_j}-x_j\partial_{x_i})\vs f)\\
&=-\Gamma_M(\vs f)=0,
\end{align*}
where the last equality follows from Remark \ref{rem:circular}. 
In a similar way, we also get $\sC(f'_s)=0$ and we conclude that $f$ belongs to the kernel of the Casimir operator $\sC$. Since $\sC$ and $r$ anticommute with each other, it also holds $\tGamma f=-r(\sC f)=0$. 
\end{proof}


For any hypercomplex basis $\B=(1,v_1,\ldots,v_n)$ of $M$, with coordinates $x_0,\ldots,x_n$, every homogeneous polynomial with coefficients in $\Aa$ in the $n+1$ real variables $x_0,\ldots, x_n$ defines a smooth function on $M$. We will denote by $\Aa_k[x_0,\ldots,x_n]$ the space of $k$-homogeneous polynomials with coefficients in $\Aa$.

\begin{proposition}\label{pro:slicepoly}
Assume $\sum_{i=1}^nk_i=(1-n)/2$. Let $k \in \mathbb{N}$ and let $f \in \Aa_k[x_0,\ldots,x_n]$ be a $k$-homo\-geneous polynomial. If $\sC f=0$ or $\tGamma f=0$, then there exist $a_{\alpha,\beta}\in\Aa$ ($\alpha,\beta\in\nn$) such that
\begin{equation}\label{eq:f}
f=\sum_{\alpha+\beta=k}x^\alpha (x^c)^\beta a_{\alpha,\beta}.  
\end{equation}
In particular, $f\in\SL(M)$, i.e., it is a slice function on $M$.
\end{proposition}
\begin{proof}
Assume that $\sC f=0$. Write $f$ as $f=x_0^kQ_0+\cdots+x_0^{k-\ell}Q_\ell+\cdots+Q_k$, where $Q_\ell$ is a $\ell$-homogeneous polynomial in $x_1,\ldots,x_n$. 
Since $\sC$ is a 0-degree operator that preserve homogeneity w.r.t.\ $x_1,\ldots,x_n$, it holds $0=\sC f=\sum_{\ell=0}^kx_0^{k-\ell}\sC Q_\ell$ if and only if $\sC Q_\ell=0$ for every $\ell=0,\ldots,k$. 
Using \eqref{casimir}, the equality $\sC Q_\ell=0$ is equivalent to
\[
\Eu(Q_\ell)=\ell Q_\ell=-\x\,\DD_\B Q_\ell.
\]
Since $\sC(\DD_\B Q_\ell)=-\DD_\B(\sC Q_\ell)=0$ and $\DD_\B Q_\ell$ is $(\ell-1)$-homogeneous (see \cite[Lemma\ 2.9]{Rosler}), we also have 
\[
(\ell -1)\DD_\B Q_\ell=-\x\, \DD_\B^2 Q_\ell, 
\]
whence $\ell(\ell-1)Q_\ell=(-\x)^2\DD_\B^2 Q_\ell$. Iterating this computation, we finally get
\[
\ell!Q_\ell=(-\x)^\ell\DD_\B^\ell Q_\ell=\x^\ell a\text{\quad for a constant }a\in\Aa.
\]
Therefore $f$ has the form
\[
f=x_0^ka_0+\cdots +x_0^{k-\ell}\x^\ell a_\ell+\cdots+\x^k a_k.
\]
Substituting $(x+x^c)/2$ for $x_0$ and $(x-x^c)/2$ for $\x$, we see that $f$ can be expressed as
\[
f=\sum_{\alpha+\beta=k}x^\alpha (x^c)^\beta a_{\alpha,\beta}
\]
for some $a_{\alpha,\beta}\in\Aa$. 
The statement for $\tGamma$ follows from the equality of the kernels of $\tGamma$ and $\sC$, a consequence of \eqref{eq:r}. 
\end{proof}

The previous results allow us to extend to any hypercomplex subspace the characterization of sliceness for polynomials on Clifford algebras obtained  in \cite[Proposition 3.3]{binosi2025dunklapproachsliceregular}. We observe that here we do not need any condition of positivity or non-singularity for the multiplicities $k_i$ (we refer to \cite{binosi2025dunklapproachsliceregular} for this concept). The unique assumption we need is that $\kappa=\sum_{i=1}^nk_i=(1-n)/2$. 

\begin{theorem}\label{teo:poly_sliceness}
Assume that the Dunkl multiplicities $k_i\in\rr$ satisfy $\sum_{i=1}^nk_i=(1-n)/2$. Let $f\in \Aa[x_0,\ldots,x_n]$ be a 
polynomial function with coefficients in $\Aa$. Then the following conditions are equivalent:
\begin{itemize} 
  \item[(i)] $f\in\SL(M)$, a slice polynomial;
  \item[(ii)] $f$ is in the kernel of the Casimir operator $\sC$; 
  \item[(iii)] $f$ is in the kernel of the spherical Dunkl-Dirac operator $\tGamma$. 
\end{itemize}
\end{theorem}
\begin{proof}
Decompose $f\in \Aa[x_0,\ldots,x_n]$ as
\[
f=x_0^mP_m+x_0^{m-1}P_{m-1}+\cdots+x_0P_1+P_0
\]
with $P_i$ polynomials in $x_1,\ldots,x_n$. Then $\sC f=\sum_{i=0}^m x_0^{i}\sC P_i=0$ if and only if $\sC P_i =0$ for every $i$. Decompose every $P_i$ into homogeneous components $P_i=P_{i,0}+\cdots+P_{i,d_i}$ and observe that $\sC P_i =0$ if and only if $\sC P_{i,j}=0$ for every $j$. We can then apply Propositions \ref{pro:kernels} and \ref{pro:slicepoly} to obtain the thesis.  
\end{proof}

Theorem \ref{teo:poly_sliceness} and Proposition \ref{pro:slicepoly} show in particular that every (left) slice polynomial on $M$ is of the form 
$\sum_{\alpha,\beta}x^\alpha (x^c)^\beta a_{\alpha,\beta},$ with $a_{\alpha,\beta}\in\Aa$.
As in \cite{binosi2025dunklapproachsliceregular}, the result extends 
to real analytic functions on an axially symmetric open set $\OO\subseteq M$. 
Our aim is to extend the characterization of sliceness of Theorem \ref{teo:poly_sliceness} to any $C^1$ function. Let $\OO\subseteq M$ be axially symmetric. We set
\[
\wsC:=\sum_{i=1}^nx_i T_{\B,i}-\Eu,\quad \sC':=\tfrac12\wsC\left(1+r\right),
\quad \sC'':=\sC'\x.
\]
These operators act on functions $f\in C^1(\OO,\Aa)$ as follows:
\begin{equation}\label{eq:tildeS}
  \wsC f=\sum_{i=1}^nx_i (T_{\B,i}f-\partial_{x_i}f)=\sum_{i=1}^nk_i(f-r_if),
\end{equation}
\begin{equation}\label{eq:S12}
  \sC' f= \wsC(\vs f),\quad \sC'' f=\wsC(\vs{(\x f)}),
\end{equation}
where $\vs f=\frac12\left(f+rf\right)$ is the spherical value of $f$ (not necessarily slice) defined in \S\ref{sec:slice}. 

\begin{proposition}\label{pro:sliceoperators}
Let $f\in\SL^1(\OO)$ be a $C^1$ slice function. Then $f\in\ker\sC'\cap\ker\sC''$. 
\end{proposition}
\begin{proof}
If $f$ is slice, then $\vs f$ and $\vs{(\x f)}$ are constant on the spheres $\s_x\cap M$ for every $x\in\OO\setminus\rr$ (see Remark \ref{rem:sliceness_continuity_M}). Therefore if $g$ is $\vs f$ or $\vs{(\x f)}$, it holds $g=r_ig$ for every $i=1,\ldots,n$ and $\sC' f= \wsC(\vs f)=0$, $\sC'' f=\wsC(\vs{(\x f)})=0$.
\end{proof}

For the converse statement, we need all the three differential-difference operators $\sC$, $\sC'$ and $\sC''$, with non-positive multiplicities $k_i$. We can assume that $M$ has dimension $n+1>2$, since in the case $n=1$ every function on $\OO$ is a slice function (see \S\ref{sec:slice}).

\begin{theorem}\label{teo:C1_sliceness}
Let $n\ge2$. Assume that the Dunkl multiplicities $k_i$ are non-positive, with at most one vanishing, and that $\sum_{i=1}^nk_i=(1-n)/2$. Let $\OO\subseteq M$ be open and axially symmetric and $f\in C^1(\OO,\Aa)$. Then the following conditions are equivalent:
\begin{itemize} 
  \item[(i)] $f$ is a slice function on $\OO$;
  \item[(ii)] $f$ is in the kernel of the operators $\sC$, $\sC'$ and $\sC''$. 
\end{itemize}
\end{theorem}

In order to prove the theorem, we begin with two Lemmas. Before their statements we introduce a notation for the composition of reflections.

Let $L$ be the isomorphism $\rr^{n+1}\simeq M$ associated to the hypercomplex basis $\B$.
Given $x\in M$ and a subset $H$ of $[n]:=\{1,\ldots,n\}$, we denote by $\bar x^H$ the image under $L$ of the $(n+1)$-tuple obtained by $(x_0,\ldots,x_n)$ by reflecting the components in $H$. The $i$-th component of $L^{-1}(\bar x^H)$ is $-x_i$ if $i\in H$, $x_i$ otherwise. If $H=\{i\}$, we also write $\bar x^i$ instead of $\bar x^{\{i\}}$. In particular, it holds $\bar x^\emptyset=x$, $\bar x^{[n]}=x^c$ and $\bar x^{H^c}=(\bar x^H)^c$ for every $H\subseteq [n]$, where $H^c=[n]\setminus H$. 

\begin{lemma}\label{PerronFrobenius}
Let $n\ge2$. Assume that the multiplicities $k_i$ are non-positive, with at most one vanishing. Let $\OO\subseteq M$ be 
$\zz_2^n$-invariant (i.e., $\bar x^i\in\OO$ for every $x\in\OO$ and every $i=1,\ldots,n$) and let $g:\OO\to\Aa$ be a function invariant w.r.t.\ conjugation $x\mapsto x^c$. If $\wsC g=0$, then
$g=r_ig$ for every $i=1,\ldots,n$, i.e., $g$ is $\zz_2^n$-invariant.
\end{lemma}
\begin{proof}
Given a point $x\in\OO$, consider the reflected points 
\[
\bar x^i=L(x_0,x_1,\ldots,-x_i,\ldots,x_n)\in\OO,\quad\text{$i=1,\ldots,n$}.
\]
We must show that $g(x)=g(\bar x^i)$ for every $i$. Since $g\in\ker \wsC$, thanks to \eqref{eq:tildeS}, it holds
\[
g\sum_{i=1}^n k_i=\sum_{i=1}^nk_i \,r_i g.
\] 
Equivalently, for every $x\in\OO$ it holds
\begin{equation}\label{eq:g}
  g(x)=\sum_{i=1}^n\alpha_i\, g(\bar x^i),\quad\text{where\quad}\alpha_i:=\tfrac{k_i}{\sum_{j=1}^n k_j}.
\end{equation}
The $\alpha_i$'s are all positive except at most one, which vanishes, and their sum is 1.

If $n=2$, equality \eqref{eq:g} evaluated at points $\bar x^1$ and $\bar x^2$ of $\OO$ yields (taking into account our hypothesis $g(x^c) = g(x)$ and the equality $\alpha_1 + \alpha_2 = 1$)
\begin{align*}
g(\bar x^1)&=\alpha_1 g(x)+\alpha_2 r_2 g(\bar x^1)=\alpha_1 g(x)+\alpha_2 g(x^c)=g(x),\\
g(\bar x^2)&=\alpha_1 r_1g(\bar x^2)+\alpha_2 g(x)=\alpha_1 g(x^c)+\alpha_2 g(x)=g(x),
\end{align*}
and the thesis is proved. 

Assume now $n>2$. Let $\widetilde\OO=\{x\in\OO\;|\; x_i\ne0 \textrm{ for all }i=1,\ldots,n\}$ and let $x\in\widetilde\OO$ be fixed.  
Let $i_0\in\{1,\ldots,n\}$ be such that $\alpha_{i}>0$ for every $i\ne i_0$. 
Consider the set $\mathcal V$ of $2^{n-1}$ points of $\OO$ defined as follows
\[
\mathcal V=\left\{\bar x^H\;\big|\;H\subseteq [n],\ H\not\ni i_0\right\}.
\]
Note that for every $H$, exactly one of $\bar x^H$ and $(\bar x^{H})^c$ belongs to $\mathcal V$. 
Equality \eqref{eq:g} evaluated at the $2^{n-1}$ points $\bar x^H\in\mathcal V$ gives 
\begin{equation}\label{eq:system}
g(\bar x^H)-\sum_{i=1}^n\alpha_i g(\bar x^{H\Delta\{i\}})=0,  
\end{equation}
where $H\Delta K$ is the symmetric difference of two sets $H,K$. 
If $\bar x^{H\Delta\{i\}}\not\in\mathcal V$, we can replace the value $g(\bar x^{H\Delta\{i\}})$ in \eqref{eq:system} with the equal value $g((\bar x^{H\Delta\{i\}})^c)$, and $(\bar x^{H\Delta\{i\}})^c\in\mathcal V$ (thanks to our hypothesis $g(x^c) = g(x)$).

From \eqref{eq:system} we see that the element of $\mathbb{A}^{2^{n-1}}$
obtained by ordering the elements of $g(\mathscr{V})$ is a solution to a homogeneous system of $2^{n-1}$ linear equations. The $2^{n-1} \times 2^{n-1}$ real matrix associated to this system takes the form $I-A$, where $I$ denotes the identity matrix and every row of $A$ vanishes, apart from $n$ entries equal to $\alpha_1, \ldots, \alpha_n$, respectively.

We prove that $A=[a_{hk}]$ is a symmetric matrix. If $\bar x^H$ is the $h$-th element in the ordering of $\mathcal V$, and $\bar x^{H\Delta\{i\}}$ (or its conjugate $(\bar x^{H\Delta\{i\}})^c$ if $\bar x^{H\Delta\{i\}}\not\in\mathcal V$) is the $k$-th element, then $a_{h,k}=\alpha_i$. But it also holds 

\begin{equation}\label{eq:system2}
g(\bar x^{H\Delta\{i\}})-\sum_{i=1}^n\alpha_i g(\bar x^{H})=0,  
\end{equation}
since $(H\Delta\{i\})\Delta\{i\}=H$, whence $a_{kh}=\alpha_i=a_{hk}$. 

So $A$ is a real, symmetric, doubly stochastic matrix of order $2^{n-1}$. We want to show that $A$ is irreducible. 
This is equivalent to proving that the graph with adjacency matrix $B$ obtained by $A$ setting $\tilde b_{hk}=1$ if $a_{hk}\ne0$ and $b_{hk}=0$ if $a_{hk}=0$, is strongly connected (see e.g.\ \cite{MeyerBook}). 
The $h$-th and $k$-th nodes (corresponding to $\bar x^H$ and $\bar x^{H\Delta\{i\}}$ (or its conjugate $(\bar x^{H\Delta\{i\}})^c$) respectively) of the graph are connected by an edge if and only if $a_{hk}=\alpha_i\ne0$. 
If $H,K\subseteq[n]\setminus\{i_0\}$, there exist a sequence of edges leading from the node corresponding to $\bar x^H\in\mathcal V$ to the node corresponding to $\bar x^{H\Delta K}\in\mathcal V$. Since every $H'\subseteq[n]\setminus\{i_0\}$ can be expressed as $H'=H\Delta(H'\Delta H)$, this means that the graph is strongly connected.

The $2^{n-1}$-tuple with elements $g(\bar x^H)$ with $H\in\mathcal V$ is an eigenvector of $A$ with eigenvalue 1.  The Perron-Frobenius Theorem (see e.g.\ \cite{MeyerBook}) applied to $A$ implies that this eigenvector is a multiple of $(1,1,\ldots,1)$. In particular, $g(x)=g(\bar x^\emptyset)=g(\bar x^i)$ for every $i=1,\ldots,n$, $i\ne i_0$. Moreover, $g(\bar x^{i_0})=g(\bar x^{\{i_0\}^c})=g(x)$. Since $\widetilde\OO$ is dense in $\OO$, by continuity the thesis is proved.
\end{proof}

\begin{lemma}\label{lemma:KerGammaM}
Let $n\ge2$. Let $\OO\subseteq M$ be open and $\zz_2^n$-invariant and let $g:\OO\to\Aa$ be a $C^1$ function such that $g=r_ig$ for every $i=1,\ldots,n$. If $\Gamma_M g=0$, then  for every $x\in\OO\setminus\rr$, $g$ is locally constant on the set $\s_x\cap\OO$.
\end{lemma}
\begin{proof}
If $n=2$, $\Gamma_M g=0$ is equivalent to $v_1(v_2 L_{12}g)=0$ or $L_{12}g=0$. Since $L_{12}$ is a generator for the tangential differential operators on the circles $\s_x\cap M$, $x\in\OO\setminus\rr$, we obtain the thesis.

Let $n>2$. Since $g$ is $\zz_2^n$-invariant, it holds $g(x)=g(\bar x^i)$ for every $x\in\OO$ and $i=1,\ldots,n$, where $\bar x^i$ is the reflected point already defined in the proof of Lemma \ref{PerronFrobenius}. 
Let $y\in\OO\setminus\rr$ and choose $z$ in the connected component of $\s_y\cap\OO$ containing $y$, with $z\ne y,y^c$. Let $\B=(1,v_1,v_2,\ldots)$ be a hypercomplex basis of $M$ such that $y,z\in\langle 1,v_1,v_2\rangle$. Since $x_j=0$ on $\langle 1,v_1,v_2\rangle$ for every $j>2$, it holds\
\begin{equation}\label{eq:sums}
(\Gamma_M g)_{|\langle 1,v_1,v_2\rangle}=(\Gamma_\B g)_{|\langle 1,v_1,v_2\rangle}=
-v_1(v_2 L_{12}g))-\sum_{j=3}^n(v_1(x_1v_j\partial_{x_j}g)-v_j(x_2v_2\partial_{x_j}g)).
\end{equation}
The $\zz_2^n$-invariance of $g$ implies that
\[
\partial_{x_i}g(\bar x^j)=
\begin{cases}\partial_{x_i}g(x)&\text{\quad if }i\ne j,\\-\partial_{x_j}g(x)&\text{\quad if }i=j.\end{cases}
\]
It follows that
\begin{align*}
\Gamma_M g(\bar x^\ell)&=-v_1(v_2 L_{12}g(x)))-\sum_{j\ne\ell,j=3}^n(v_1(x_1v_j\partial_{x_j}g(x))-v_j(x_2v_2\partial_{x_j}g(x)))\\
&\quad+v_1(x_1v_\ell\partial_{x_\ell}g(x))-v_\ell (x_2v_2\partial_{x_\ell}g(x))
\end{align*}
for every $\ell=3,\ldots,n$. Let $s$ denote the last sum in \eqref{eq:sums} evaluated at $x\in\OO\cap\langle 1,v_1,v_2\rangle$. Then
\begin{align*}
0&=\Gamma_M g(x)=-v_1(v_2 L_{12}g(x)))-s\text{\quad and}\\
0&=\sum_{\ell=3}^n\Gamma_M g(\bar x^\ell)=-(n-2)v_1(v_2 L_{12}g(x))-(n-3)s,
\end{align*}
which yields $v_1(v_2 L_{12}g(x))=0$. Then $L_{12}g$ vanishes on $\OO\cap\langle 1,v_1,v_2\rangle$. Since $L_{12}$ is a generator for the tangential differential operators on the circles $\s_x\cap \langle 1,v_1,v_2\rangle$, $x\in\OO\setminus\rr$, we obtain that $g(y)=g(z)$, i.e., $g$ is locally constant on $\s_x\cap\OO$. 
\end{proof}

In the previous proof it is crucial that the spherical Dirac operator $\Gamma_M$ does not depend on the hypercomplex basis $\B$ used in its definition. 

\begin{proof}[Proof of Theorem \ref{teo:C1_sliceness}]
The implication $(\textrm{i})\Rightarrow(\textrm{ii})$ follows from Propositions \ref{pro:kernels} and \ref{pro:sliceoperators}. 
Conversely, assume that $\sC f=\sC' f=\sC''f=0$. Let $g=\vs f$ or $g=\vs{(\x f)}$. Then $g$ satisfies the assumptions of Lemma \ref{PerronFrobenius}. Therefore $g$ is $\zz_2^n$-invariant whence $T_{\B,i}g=\partial_{x_i}g$ for $i=1,\ldots,n$ and $\DD_\B g=\sum_{i=1}^nv_i\partial_{x_i}g$. Using \eqref{casimir} and \eqref{eq:anticommute}, it follows that 
\begin{align*}
\sC g&=\x\,\DD_\B g+\Eu g=\sum_{i=1}^nx_iv_i\left(\sum_{j=1}^nv_j \partial_{x_j}g\right)+\sum_{i=1}^nx_i\partial_{x_i}g\\
&=\sum_{i< j}(x_i v_i(v_j \partial_{x_j}g)+x_jv_j(v_i\partial_{x_i}g))
=\sum_{i<j}v_i(v_j(x_i \partial_{x_j}-x_j\partial_{x_i})g)=-\Gamma_M g.
\end{align*}
Using the commutation relation 
\eqref{eq:r} and Remark \ref{rem:circular}, when $g=\vs f$, we get
\[
2\Gamma_M(\vs f)=-2\sC(\vs f)=-\sC\left(f+rf\right)=-\sC f+ r\sC f=0.
\]
When $g=\vs{(\x f)}$, we obtain, again from Remark \ref{rem:circular}, that $\Gamma_M(\vs{(\x f)})=0$. 
We can then apply Lemma \ref{lemma:KerGammaM} and get that $\vs f$ and $\vs{(\x f)}$ are constant on every sphere $\s_x\cap M$, with $x\in \OO\setminus\rr$. 
Thanks to Proposition \ref{pro:sliceness_continuity_M}, this means that $f$ is a slice function on $\OO$.
\end{proof}

\subsection{Dunkl-Dirac monogenicity and slice regularity} 
\label{sub:dirac_dunkl_monogenicity_and_slice_regularity}

Let $\B=(1,v_1,\ldots,v_n)$ be a hypercomplex basis of $M$, with coordinates $x_0,\ldots,x_n$. 

\begin{proposition}\label{pro:equality}
Assume that the Dunkl multiplicities have sum $\kappa=\sum_{i=1}^nk_i=(1-n)/2$. Let $\OO\subseteq M$ be open and axially symmetric and $f\in C^1(\OO,\Aa)$. If $\sC f=0$, then $D_\B f=\difbarM f$. 
\end{proposition}

\begin{proof}
If $\sC f=0$, from \eqref{casimir} we obtain, for every $x\in\OO\setminus\rr$,
\[
-\x^{-1}\Eu f=\DD_\B f\ \Rightarrow\ \difbarM f=\partial_{x_0} f-\x^{-1}\Eu f=\partial_{x_0} f+\DD_\B f=D_\B f.
\]
\end{proof}

From Propositions \ref{pro:kernels} and \ref{pro:equality} we get the following result. 

\begin{corollary}\label{cor:sr_are Dunkl}
Let $k_i\in\rr$ be Dunkl multiplicities such that $\sum_{i=1}^nk_i=(1-n)/2$. Then every slice-regular function $f\in\sr(\OO)$ is Dunkl monogenic, i.e., $D_\B f=0$. \hfill\qed
\end{corollary}

As a consequence of Corollary \ref{cor:sr_are Dunkl}, every slice-regular function is Dunkl harmonic, i.e., in the kernel of the Dunkl-Laplace operator $\Delta_{D,M}$ on $M$ (see \S\ref{subsec:Dunkl-Dirac_operators})  associated with the reflection group $\zz_2^n$ with sum of the multiplicities equal to $(1-n)/2$:
\[
\Delta_{D,M}=\Delta_M+\sum_{i=1}^nk_i\left(\frac2{x_i}\partial_{x_i}-\frac{1-r_i}{x_i^2}\right)
\]
(see e.g.\ \cite[\S7.5.1]{DunklXu}). 
Therefore, if $f\in\sr(\OO)$ is slice-regular, it holds (see \eqref{eq:db} and \cite[Prop.\ 9]{CRoperators})
\begin{align*}
\dM f&=(1-n)\sd f=\sum_{i=1}^n \frac{k_iv_i}{x_i}(r_i f-f),\quad
\Delta_M f=\sum_{i=1}^nk_i\left(\frac{f(x)-r_if(x)}{x_i^2}-\frac2{x_i}\dd{f}{x_i}\right).
\end{align*}

The previous results allow us to extend to any hypercomplex subspace the characterization of slice-regularity for polynomials on Clifford algebras \cite[Proposition 3.3]{binosi2025dunklapproachsliceregular}. 

\begin{theorem}\label{teo:poly_slice_regularity}
Assume that the Dunkl multiplicities $k_i$ satisfy $\sum_{i=1}^nk_i=(1-n)/2$. Let $f\in \Aa[x_0,\ldots,x_n]$ be a 
polynomial function with coefficients in $\Aa$. The following conditions are equivalent:
\begin{itemize} 
  \item[(i)] $f$ is a slice-regular polynomial, namely, it has the form $f(x)=\sum_{j=0}^m x^ja_j$, with coefficients $a_j\in\Aa$;
  \item[(ii)] $f$ is Dunkl monogenic and belongs to the kernel of the Casimir operator $\sC$ (or of the spherical Dunkl-Dirac operator $\tGamma$).
\end{itemize}
\end{theorem}
\begin{proof}
The implication (i) $\Rightarrow$ (ii) follows from Corollary \ref{cor:sr_are Dunkl} and Proposition \ref{pro:kernels}. The converse implication comes from Theorem \ref{teo:poly_sliceness}, Proposition \ref{pro:equality}  and Proposition \ref{pro:sliceregularpoly}, using the fact that a slice polynomial $f$ is slice-regular if and only if $\difbarM f=0$ (see \cite[Theorem\ 2.4]{Gh_Pe_GlobDiff}). 
\end{proof}

Now we generalize the result to $C^1$ functions.

\begin{theorem}\label{teo:slice_regularity}
Assume that the Dunkl multiplicities $k_i$ are non-positive, with at most one zero multiplicity, and that $\sum_{i=1}^nk_i=(1-n)/2$. 
Let $\OO\subseteq M$ be open and axially symmetric and $f\in C^1(\OO,\Aa)$. The following conditions are equivalent:
\begin{itemize} 
  \item[(i)] $f$ is slice-regular on $\OO$;
  \item[(ii)] $f$ is Dunkl monogenic and belongs to the kernel of the operators $\sC$, $\sC'$ and $\sC''$. 
\end{itemize}
\end{theorem}

\begin{proof}
The implication (i) $\Rightarrow$ (ii) follows from Propositions \ref{pro:kernels} and \ref{pro:sliceoperators}. Conversely, assume that $D_\B f=\sC f=\sC' f=\sC''f=0$. From Theorem \ref{teo:C1_sliceness} and Proposition \ref{pro:equality} it follows that $f$ is a slice function on $\OO$, with $\difbarM f=D_\B f=0$. In view of \cite[Theorem\ 2.2]{Gh_Pe_GlobDiff}, $f$ is slice-regular on $\OO$.
\end{proof}

\begin{remark}
If $n=1$, the condition on the multiplicities implies $k_1=0$. Both the operators $D_\B$ and $\difbarM$ coincide with the Cauchy-Riemann operator of $M\simeq\cc$ and the Casimir operator $\sC$, as well as the operators $\tGamma$, $\sC'$ and $\sC''$,  is the zero operator. In this case slice-regularity and Dunkl-monogenicity both coincide with complex holomorphy. When $n>1$, these classes of functions in general are different. 
\end{remark}

\section{Dunkl-regular function spaces} 
\label{sec:dunkl_regular_function_spaces}

\subsection{Intermediate Dunkl-Dirac operators} 
\label{sub:Intermediate_Dunkl-Dirac_operators}

Let $\B=(1,v_1,\ldots,v_n)$ be a hypercomplex basis of the hypercomplex subspace $M$ of $\Aa$, with associated coordinates $x_0,x_1,\ldots,x_n$.  Following \cite{DeBieGenestVinet}, which developed the Clifford algebra case, we introduce the intermediate $\zz_2^n$ Dunkl-Dirac operators of $M$ w.r.t.\ $\B$. 

A subset $\OO$ of $M$ will be called \emph{$A$-symmetric} if is invariant w.r.t.\ reflections $r_i$ for all $i\in A$.

\begin{definition}\label{def:DiracDunklA}
Let $\OO\subseteq M$ be open. Let $A\subseteq[n]=\{1,\ldots,n\}$ and suppose that $\OO$ is $A$-symmetric. We define
$\DD_A:C^1(\OO,\Aa)\to C^0(\OO,\Aa)$ as 
\[
\DD_A=\sum_{i\in A}v_i T_{\B,i},
\]
where $T_{\B,i}$ are the Dunkl operators introduced in Definition \ref{def:DiracDunkl}. We also set 
\[
D_A=\partial_{x_0}+\sum_{j\in[n]\setminus A}v_j\partial_{x_j}+\DD_A.
\]
\end{definition}
Observe that the operator $D_A$ is the 
Dunkl-Cauchy-Riemann operator of $M$ with multiplicities $k_i=0$ for every $i\not\in A$.  In particular, $\DD_{[n]}$ is the operator $\DD_{\B}$, while $D_\emptyset=\dM$. 

We set $\x_A=\sum_{i\in A}v_i x_i$, $x_A=x_0+\x_A$. 
In particular, it holds $\x_{[n]}=\x$ and $x_\emptyset=0$.
Proposition \ref{pro:orsted} generalizes to the intermediate operators $\DD_A$.
The arguments of the proof of Proposition \ref{pro:orsted} can be repeated taking the subset $A\subseteq[n]$ in place of $[n]$. If $A\ne\emptyset$, it can be shown that the commutation relations can be completed to get that also the intermediate operators realize the Lie superalgebra $\mathfrak{osp}(1|2)$  (see \cite[Proposition 1]{DeBieGenestVinet} for the Clifford algebra case).

\begin{proposition}\label{pro:orstedA}
Let $\gamma_A:=|A|/2+\sum_{i\in A}k_i$, where $|A|$ is the cardinality of $A$ and let $\Delta_{A}=\sum_{i\in A}T^2_{\B,i}$ and $\Eu_A=\sum_{i\in A}x_i\partial_{x_i}$. It holds 
\begin{align*}
&\left\{\x_A,\x_A\right\}=-2\|\x_A\|^2;& &\left\{\DD_A,\DD_A\right\}=-2\Delta_{A}; && \left[\Eu_A,\x_A\right]=\x_A;\\
&\left\{\x_A,\DD_A\right\}=-2(\Eu_A+\gamma_A);& &\left[\DD_A,\Eu_A+\gamma_A\right]=\left[\DD_A,\Eu_A\right]=\DD_A; && \left[\Delta_{A},\x_A\right]=2\DD_A;\\
& \left[\DD_A,\|\x_A\|^2\right]=2\x_A,&&
\end{align*}
\end{proposition}\hfill\qed

\begin{definition}
Let $A\subseteq[n]$. Let $\OO\subseteq M$ be open and $A$-symmetric. 
We define the operators 
$\SA,\GammaA:C^1(\OO,\Aa)\to C^0(\OO,\Aa)$ as 
\[
\SA=\tfrac12\left(\left[\x_A,\DD_A\right]-1\right),\quad \GammaA=\SA r_A,
\]
where $r_A$ is the composition of reflections $r_i$, $i\in A$. We also set
\[
\wSA:=\sum_{i\in A}x_i T_{\B,i}-\Eu_A=\sum_{i\in A}k_i(1-r_i),\quad \SA':=\tfrac12\wSA\left(1+r_A\right),
\quad \SA'':=\SA'\x_A.
\]
and $\mathscr S_A
:C^1(\OO,\Aa)\to (C^0(\OO,\Aa))^3$ the real-linear operator defined as follows
\[
\mathscr S_{A} f=(\SA f,\SA' f,\SA'' f).
\]
\end{definition}

\begin{remark}\label{rem:commutation}
As in Proposition \ref{pro:operators}, it can be proved that it holds $\{\SA,\DD_A\}=\{\SA,\x_A\}=0$ and  $\{\DD_A,r_A\}=\{\SA,r_A\}=\{\x_A,r_A\}=0$. 
\end{remark}

Observe that for $A=\emptyset$, it holds 
\[
\underline S_\emptyset=\widetilde\Gamma_\emptyset=-1/2,\quad \underline{\widetilde S}_\emptyset=\underline{S}_\emptyset'=\underline{S}''_\emptyset=0,
\]
while for $A=\{i\}$, a direct computation gives
\[
\underline S_{\{i\}}=k_ir_i,\quad \widetilde\Gamma_{\{i\}}=k_i,\quad \underline{\widetilde S}_{\{i\}}=k_i(1-r_i),\quad \underline{S}_{\{i\}}'=\underline S''_{\{i\}}=0.
\]

In the following, for any $A\subseteq[n]$, we will always assume that the Dunkl multiplicities $k_i$ in $\DD_A$ and $\SA$ are non-positive, with $k_i=0$ for at most one index $i\in A$ and $\sum_{i\in A}k_i=(1-|A|)/2$, which is equivalent to $\gamma_A=1/2$. 
Under these conditions the intermediate Casimir operator $\SA$ can be written, as in Proposition \ref{pro:operators}, in the form
\begin{equation}\label{eq:SA}
  \SA=\x_A\DD_A+\Eu_A,
\end{equation}
and it holds $\mathscr S_{\{i\}}=0$ for every $i=1,\ldots,n$.  

\begin{definition}
Let $\emptyset\not=A=\{i_1,\ldots,i_\ell\}\subseteq[n]$ and let $\OO\subseteq M$ be $A$-symmetric. 
Given a function $f:\OO\to\Aa$, we define the \emph{$A$-spherical value} of $f$ as the function $f^\circ_{s,A}:\OO\to\Aa$ defined by
\[
f^\circ_{s,A}(x)=\tfrac12(f(x)+f(\bar x^A)).
\]
Here $\bar x^A=r_A(x)=(r_{i_1}\circ\cdots\circ r_{i_\ell})(x)\in\OO$ is the reflected point w.r.t.\ $A$.
\end{definition}

\begin{remark}\label{rem:sliceness_continuity_A}
By definition, the equality $f=f^\circ_{s,A}+\x_A^{-1}{(\x_A f)}^\circ_{s,A}$ holds on $\OO\setminus\{\x_A=0\}$.
As in Proposition \ref{pro:sliceness_continuity_M}, if $f\in C(\OO,\Aa)$, the function $g:=\|\x_A\|^{-1}(\x_Af)^\circ_{s,A}$ extends continuously to $\OO$ with value zero on $\OO\cap\{\x_A=0\}$: if  $y\in\OO$ with $\underline y_A=0$, then 
\[
\|g(x)\|\le \omega\|\x_A^{-1}{(\x_A f)}^\circ_{s,A}(x)\|=\omega\|f(x)-f^\circ_{s,A}(x)\|\le\omega(\|f(x)-f(y)\|+\|f(y)-f^\circ_{s,A}(x)\|)
\to0
\]
as $x$ tends to $y$. 
\end{remark}

\begin{remark}\label{rem:fsa}
If $f\in\SL(\OO)$ is a slice function, then its $A$-spherical value is related to $\vs f$ and $f'_s$ by the formula
\begin{equation}\label{eq:equality_spherical_slice}
\textstyle
  f^\circ_{s,A}(x)=\vs f(x)+\tfrac12\IM(x+\bar x^A) f'_s(x)=\vs f(x)+(x-x_A)f'_s(x),
\end{equation}
valid for every $x\in\OO\setminus\rr$. This is a direct consequence of formula \eqref{eq:spherical}.
A similar computation yields 
\begin{equation}\label{eq:equality_spherical_derivative_slice}
\textstyle
  (\x_A f)^\circ_{s,A}(x)=\x_A^2 f'_s(x)=-\|\x_A\|^2 f'_s(x).
\end{equation}
In particular, if $f$ is a slice function, $f^\circ_{s,A}$ and $(\x_A f)^\circ_{s,A}$ are $r_i$-invariant for every $i\in A$.
\end{remark}

Let $A=\{i_1,\ldots,i_\ell\}\subseteq[n]$. Let $M_A$ be the hypercomplex subspace $M_A:=\langle 1,v_{i_1},\ldots,v_{i_\ell}\rangle\subseteq M$ and $\s_{M_A}=\s_\Aa\cap M_A$ the $(\ell-1)$-dimensional unit sphere in $M_A\cap\ker(t)$.

\begin{definition}
Let $\emptyset\ne A\subseteq[n]$ and $A^c=[n]\setminus A=\{j_1,\ldots,j_{n-\ell}\}$ with $j_1<\cdots <j_{n-\ell}$. A subset $\OO$ of $M$ is called \emph{$A$-circular} if for any $x=x_0+\x_A+\x_{A^c}\in\OO$, with $0\ne\x_A=\beta I$ and $I\in \s_{M_A}$, $\beta\in\rr$,  the element $x_J:=x_0+\beta J+\x_{A^c}$ of $M$ belongs to $\OO$ for every $J\in \s_{M_A}$. 
Equivalently, $\OO$ is $A$-circular if there exists $E=E'\times E''\subseteq\cc\times\rr^{n-\ell}$ such that $\OO=\OO_E$, where 
\[\textstyle
\OO_E:=\left\{x_0+\beta J+\sum_{i=1}^{n-\ell}v_{j_i}y_i\;|\; (x_0+i\beta,y)\in E, J\in \s_{M_A}\right\}.
\]
\end{definition}

Every $A$-circular set is $A$-symmetric (see Definition \ref{def:DiracDunklA}): if $x=x_0+\beta I+\x_{A^c}\in\OO$, with $I\in \s_{M_A}$, and $i\in A$, then $r_i(I)\in \s_{M_A}$ and  $r_i(x)=x_0+\beta r_i(I)+\x_{A^c}\in\OO$. 
Observe that a set $\OO$ is axially symmetric if and only if it is $[n]$-circular. Therefore $\OO$ is $A$-circular in $M$ if and only if $\OO\cap M_A$ is axially symmetric in $M_A$. 
Moreover, if $A=\{i\}$, then $\s_{M_A}=\{\pm v_i\}$ and $\OO$ is $A$-circular if and only if it is $A$-symmetric. 

\begin{definition}\label{def:A-Dunkl-regular}
Let $\OO\subseteq M$ be open and $\zz_2^n$-invariant and let $A$ 
be a subset of $[n]$.
We call \emph{$A$-Dunkl-regular 
functions} on $\OO$ the elements of the function space 
\[
\F_{A}(\OO):=\{f\in C^1(\OO,\Aa)\;|\; f\in\ker D_A\cap\ker \mathscr S_A\}.
\]
It is easy to verify that $\F_A(\OO)$ is a real vector space (a right $\Aa$-module if $\Aa$ is associative). 
\end{definition}

\begin{remark}\label{rem:FA}
(a)\quad Since $\F_{A}(\OO)\subseteq\ker D_A$, every $A$-Dunkl-regular function is, in particular, Dunkl monogenic on $\OO$ whence Dunkl harmonic. 

(b)\quad Since $\ker\mathscr S_\emptyset=\{0\}$, it holds $\F_\emptyset(\OO)=\{0\}$. 

(c)\quad If $A=\{i\}$, it holds 
$\mathscr S_{\{i\}}=0$ and $k_i=0$. Then $\F_{\{i\}}(\OO)=\ker D_{\{i\}}=\ker \dM$ is the space $\M(\OO)$ of monogenic functions on $\OO$ for every $i=1,\ldots,n$. 
\end{remark}

\begin{remark}\label{rem:FAsubsets}
Let $A\not=\emptyset,[n]$ and let $A^c=[n]\setminus A$. Since on $M_A$ it holds $D_A=\partial_{x_0}+\DD_A$ and on $M_{A^c}$ it holds $D_A=\dibar_{M_{A^c}}$ and $\mathscr S_A=0$ (every function is $r_i$-invariant for every $i\in A$), the inclusions of the hypercomplex subspaces $M_A\hookrightarrow M$ and $M_{A^c}\hookrightarrow M$ induce space inclusions of $\sr(\OO\cap M_A)$ and of $\M(\OO\cap M_{A^c})$ in the space $\F_A(\OO)$. 
However, observe that if $f\in\F_A(\OO)$, then it can happen that $f_{|\OO\cap M_A}\not\in\sr(\OO\cap M_A)$ and $f_{|\OO\cap M_{A^c}}\not\in\M(\OO\cap M_{A^c})$. This happens, for instance, when $f = f_1 + f_2$, with $f_1,f_2\in\F_{\{2,3\}}(\hh)$ as in the forthcoming Example \ref{ex:linear}. 
\end{remark}

More in general, one can consider two different kernels obtained from two subsets $A,B\subseteq[n]$.


\begin{definition}\label{def:FAB}
Let $\OO\subseteq M$ be open and $\zz_2^n$-invariant and let $A,B$ 
subsets of $[n]$.
We call \emph{$(A,B)$-Dunkl-regular 
functions} on $\OO$ the elements of the function space 
\[
\F_{A,B}(\OO):=\{f\in C^1(\OO,\Aa)\;|\; f\in\ker D_A\cap\ker \mathscr S_B\}.
\] 
\end{definition}

\begin{remark}\label{rem:FAB}
(a)\quad Since $\F_{A,B}(\OO)\subseteq\ker D_A$, every $(A,B)$-Dunkl-regular function is, in particular, Dunkl monogenic on $\OO$ whence Dunkl harmonic. 

(b)\quad Since $\ker\mathcal S_\emptyset=\{0\}$, it holds $\F_{A,\emptyset}(\OO)=\{0\}$ for every $A$. Moreover, $\F_{\emptyset,B}(\OO)=\ker\dM\cap\ker\mathscr S_B=\M(\OO)\cap\ker\mathscr S_B$ for every $B$. This means that every $(\emptyset,B)$-Dunkl-regular function is monogenic on $\OO$. 
\end{remark}

Using the same arguments of the proofs of Lemmas \ref{PerronFrobenius} and \ref{lemma:KerGammaM}, we obtain  intermediate versions of those results. 

\begin{lemma}\label{PerronFrobeniusA}
Let $A\subseteq[n]$. 
Let $\OO\subseteq M$ be $A$-symmetric 
 and let $g:\OO\to\Aa$ be a function invariant w.r.t.\ conjugation $x\mapsto \bar x^A$. 
If $\wSA g=0$ on $\OO$, then $g=r_ig$ on $\OO$ for every $i\in A$.\qed
\end{lemma}

\begin{lemma}\label{lemma:KerGammaMA}
Let $A\subseteq[n]$.  Let $\OO\subseteq M$ be open and $A$-symmetric. 
Let $g:\OO\to\Aa$ be a $C^1$ function such that $g=r_ig$ for every $i\in A$. 
For every $x\in \OO\setminus\rr$, consider the set 
\[
\s_{x,A}:=\{x'=x'_A+(x-x_A)\in M\;|\; x'_A\in\s_x\cap M_A\}. 
\]
If $\Gamma_{M_A} g=0$ on $\OO$, then for every $x\in\OO\setminus\rr$, $g$ is locally constant on the set $\s_{x,A}$. \qed
\end{lemma}

We recall that 
$\Gamma_{M_A}=-\sum_{i,j\in A, i<j} v_i(v_jL_{ij})$, with $L_{ij}$ tangential differential operators for the spheres $\s_x\cap M_A$ whence also for the translated spheres $\s_{x,A}$. 

The next proposition characterizes functions in the kernel of the differential-difference operator $\mathscr S_A$.

\begin{proposition}\label{pro:SA}
Let $\emptyset\ne A\subseteq[n]$ and $A^c=[n]\setminus A=\{j_1,\ldots,j_{n-\ell}\}$ with $j_1<\cdots <j_{n-\ell}$. Let $\OO=\OO_E$ be an $A$-circular open subset of $M$, with $E=E'\times E''\subseteq\cc\times\rr^{n-\ell}$.
If $f\in C^1(\OO,\Aa)$ and $\mathscr S_A f=0$ on $\OO$, then there exist continuous functions 
$F_{\emptyset}^A,F_1^A:E\to\Aa$ such that for every $x\in\OO$ it holds
\[
f(x)=F_{\emptyset}^A(z,y)+JF_1^A(z,y).
\]
Here $z=x_0+i\beta\in E'$, $y\in E''$, $x=x_0+\beta J+\sum_{i=1}^{n-\ell}v_{j_i}y_i\in\OO$, where $J=\text{sign}(\beta)\tfrac{\x_A}{\|\x_A\|}=\beta^{-1}\x_A$ if $|\beta|=\|\x_A\|\ne0$, and $J$ is any element of $\s_{M_A}$ if $\beta=0=\x_A$. The functions $F_{\emptyset}^A,F_1^A$ are, respectively, even and odd w.r.t.\ $\beta$.
Conversely, if $f(x)=F_{\emptyset}(z,y)+JF_1(z,y)$ with $F_{\emptyset},F_1$ an even/odd pair w.r.t.\ $\beta$, then 
$\mathscr S_A f=0$ on $\OO$.
\end{proposition}

\begin{proof}
Assume that  $\SA f=\SA' f=\SA'' f=0$ on $\OO$. Let $g=f^\circ_{s,A}$ or $g=(\x_A f)^\circ_{s,A}$. Then $g$ satisfies the assumptions of Lemma \ref{PerronFrobeniusA}. Therefore $g=r_ig$,  $T_{\B,i}g=\partial_{x_i}g$ for every $i\in A$ and $\DD_A g=\sum_{i\in A}v_i\partial_{x_i}g$. Using \eqref{eq:SA}, it follows that 
\begin{align}\label{eq:sagamma}
\SA g&=\x_A\,\DD_A g+\Eu_A g =\textstyle\sum_{i\in A}x_iv_i\left(\sum_{j\in A}v_j \partial_{x_i} g \right)+\sum_{i\in A}x_i\partial_{x_i} g \\\notag
&\textstyle=\sum_{i,j\in A,i<j}(x_i v_i(v_j \partial_{x_j} g )+x_jv_j(v_i\partial_{x_i} g ))\\\notag
&\textstyle=\sum_{i,j\in A,i<j}v_i(v_j(x_i \partial_{x_j}-x_j\partial_{x_i})g)=-\Gamma_{M_A} g.
\end{align}
If $g=f^\circ_{s,A}$, then 
\[
2\Gamma_{M_A}(f^\circ_{s,A})=-2\SA(f^\circ_{s,A})=-\SA\left(f+r_A f\right)=-\SA f+ r_A\SA f=0,
\]
where we used Remark \ref{rem:commutation}. The same Remark allows to obtain the same equality for $g=(\x_A f)^\circ_{s,A}$. Therefore $f^\circ_{s,A}$ and $(\x_A f)^\circ_{s,A}$ belongs to $\ker \Gamma_{M_A}$. 
From Lemma \ref{lemma:KerGammaMA} we get that, for every $x\in \OO\setminus\rr$, the functions $f^\circ_{s,A}$ and $(\x_A f)^\circ_{s,A}$ are constant on the connected sets 
\[
\s_{x,A}=\{x'=x'_A+(x-x_A)\in M\;|\; x'_A\in\s_x\cap M_A\}. 
\]
This allows to define
\[
F_{\emptyset}^A(z,y):=f^\circ_{s,A}(x),\quad F_1^A(z,y):=-\beta^{-1}(\x_Af)^\circ_{s,A}(x), 
\]
when $\beta\ne0$ and $F_1^A(z,y):=0$ when $\beta=0$, i.e., when $\x_A=0$. 
In view of Remark \ref{rem:sliceness_continuity_A}, $F_{\emptyset}^A$ and $F_1^A$ extend continuously to $\OO$ and since $\x_A^{-1}=-\beta^{-1}J$ when $\beta\ne0$, it holds $f(x)=F_{\emptyset}^A(z,y)+JF_1^A(z,y)$. If $x\in\OO$ and $\beta=0$, then $x\in M_{A^c}$ and $F_{\emptyset}^A(z,y)=f(x)$, $F_1^A(z,y)=0$. 

Now we prove the converse statement. 
Assume that $f(x)=F_{\emptyset}(z,y)+JF_1(z,y)$, with $x=x_0+\beta J+\sum_{i=1}^{n-\ell}v_{j_i}y_i\in\OO$, $J\in\s_{M_A}$, $z=x_0+i\beta\in E'$ and $y\in E''$. If $g=F_{\emptyset}(z,y)$ or $g=F_1(z,y)$, then 
$g$ is invariant w.r.t.\ $J\in \s_{M_A}$ and therefore, as in \eqref{eq:sagamma},
\[
\SA g =\x_A\,\DD_A g +\Eu_A g 
\textstyle=\sum_{i,j\in A,i<j}(x_i v_i(v_j \partial_{x_j} g )+x_jv_j(v_i\partial_{x_i} g ))=-\Gamma_{M_A} g=0.
\]
Let $f_0(x):=F_\emptyset(z,y)$, $f_1(x):=F_1(z,y)$.
Since $f(x)=F_{\emptyset}(z,y)+JF_1(z,y)=f_0(x)+\beta^{-1}\x_A f_1(x)$, 
using the equality $\{ \underline{S}_A, \underline{x}_A \} = 0$ from Remark \ref{rem:commutation}, we get $\SA f=\SA (f_0)-\x_A\SA(\beta^{-1}f_1)=0$ on $\OO$.

Since $\x_A f=\x_A F_\emptyset+\beta^{-1}\x_A^2 F_1=\x_A f_0-\beta f_1$ and $f_0$ and $f_1$ are $r_i$-invariant for all $i\in A$, it also holds $\SA' f=\wSA (f^\circ_{s,A})=\wSA(f_0)=0$ and $\SA'' f=\wSA((\x_A f)^\circ_{s,A})=\wSA(-\beta f_1)=0$ on $\OO$. Therefore $\mathscr S_A f=0$.
\end{proof}

\begin{remark}\label{rem:even-odd}
If $f(x)=F_{\emptyset}^A(z,y)+JF_1^A(z,y)$ with $F_{\emptyset}^A,F_1^A$ as in the Proposition, 
then $f^\circ_{s,A}(x)=F_{\emptyset}^A(z,y)$ and $(\x_Af)^\circ_{s,A}(x)=-\beta F_1^A(z,y)$. 
In particular, $f^\circ_{s,A}$ and ${(\x_A f)}^\circ_{s,A}$ are constant on $\s_x\cap M_A$ for every $x\in\OO\cap M_A$. Moreover, this shows that $F_{\emptyset}^A$ and $F_1^A$ are uniquely determined by $f$ and $A$.
\end{remark}

\begin{remark}
If $A$ has the form $\{n-\ell+1,\ldots,n\}$, with $\ell>0$, then a function $f$ as in Proposition \ref{pro:SA} is the \emph{$T$-function} induced by the \emph{$T$-stem function} with components $F_{\emptyset}^A, F_1^A$, a concept introduced in \cite[\S3]{GhiloniStoppatoJGP} and investigated in \cite{GhiloniStoppato_arXiv24}.
\end{remark}

We now establish two general properties of the operators $\mathscr S_A$ that will be useful in the sequel. 
 In the context of $T$-functions, next result has been proved in \cite[Proposition 4.3]{GhiloniStoppato_JMAA26}. 

\begin{proposition}\label{pro:kerS}
Let $\OO\subseteq M$ be open. If $B\subseteq A\subseteq [n]$ and $\OO$ is $A$-circular, then $\ker\mathscr S_B\supseteq\ker\mathscr S_A$. In particular, if $\OO$ is axially symmetric and $f\in\SL^1(\OO)$ is a slice function, then $\mathscr S_A f=0$ on $\OO$ for every $A\subseteq [n]$.  
\end{proposition}
\begin{proof}
If $\mathscr S_A f=0$, from Proposition \ref{pro:SA} we get
\[
f(x)=F_{\emptyset}^A(z,y)+JF_1^A(z,y)=F_{\emptyset}^A(z,y)+\beta^{-1}\x_A F_1^A(z,y)
\]
for every $x\in\OO$. Since $F_{\emptyset}^A$ and $F_1^A$ are $r_i$-invariant for every $i\in B$ and $\s_{M_B}\subseteq \s_{M_A}$, we can use the same arguments of the converse part of the proof of Proposition \ref{pro:SA} and obtain that $f\in\ker\mathscr S_B$. 
\end{proof}

\begin{proposition}[Commutation relations]\label{pro:Commutation relations}
Let $A,B\subseteq[n]$. If $A\cap B=\emptyset$, then it holds
\begin{align*}
&\big\{\x_A,\x_B\big\}=0,\ \big[\Eu_A,\x_B\big]=0,\ \big\{\DD_A,\x_B\big\}=0,\ \big[\SA,\x_B\big]=0,\ 
\big\{\DD_A,\DD_B\big\}=0,\ \big[\SA,\DD_B\big]=0,
\\
&\big[\wSA,\x_B\big]=0,\ \big[\wSA r_A,\x_B\big]=0,\ \big[\SA',\x_B\big]=0,\ \big\{\SA'',\x_B\big\}=0,\ \big[\mathscr S_A,r_B\big]=0.   
\end{align*}
\end{proposition}
\begin{proof}
The first two equalities are trivial. To prove the third, we compute
\[
\DD_A(\x_Bf)=\sum_{i\in A}v_i T_{\B,i}(\x_B f)=\sum_{i\in A}v_i (\x_B T_{\B,i}f)=-\x_B\sum_{i\in A}v_i T_{\B,i}f=-\x_B\DD_A f.
\]
The second equality is true because no $i \in A$ belongs to $B$, while the third equality follows from \eqref{eq:anticommute}.
A similar direct computation gives the equality $\{\DD_A,\DD_B\}=0$. 
Using \eqref{eq:SA} and the first three equalities, we obtain that $[\SA,\x_B]=0$ and $[\SA,\DD_B]=0$.

The first two equalities of the second row are trivial and imply the third $[\SA',\x_B]=0$. Using $\{\x_A,\x_B\}=0$, we get that $\{\wSA\x_A,\x_B\}=0$ and $\{\wSA r_A\x_A,\x_B\}=0$, which yields the equality $\{\SA'',\x_B\}=0$. The last commutation relation is trivial.
\end{proof}

\subsection{Examples: quaternions, Clifford algebras, octonions}\label{sub:examples}

\subsubsection{Quaternions}

\begin{proposition}\label{pro:Dunkl_H}
If $M=\hh$ and $\B=(1,i,j,k)$, then $n=3$ and there are five (three essentially inequivalent) 
real vector spaces of $A$-Dunkl-regular functions  on an open set $\OO\subseteq\hh$, excluding the trivial case $\F_\emptyset(\OO)=\{0\}$:
\begin{itemize}
  \item[(i)] $|A|=1$. The spaces $\F_{\{1\}}(\OO)=\F_{\{2\}}(\OO)=\F_{\{3\}}(\OO)=\ker\dibar_{\hh}$ are the space of Fueter-regular functions on $\OO$.

  \item[(ii)] $|A|=2$. Let $k_2,k_3\le0$ and $k_2+k_3=-1/2$. Let $\OO$ be $A$-circular.  
  The space $\F_{\{2,3\}}(\OO)$ is the set of strongly $(1,3)$-regular functions on $\OO$. 
  This family of functions was introduced in \cite[Example\ 3.10 and Definition\ 3.13]{GhiloniStoppatoJGP}.
  \\
  The spaces $\F_{\{1,2\}}(\OO)$ and $\F_{\{1,3\}}(\OO)$ are isomorphic to $\F_{\{2,3\}}(\OO)$ as real vector spaces. They can be obtained in the same way as $\F_{\{2,3\}}(\OO)$ starting from the hypercomplex bases $\B'=(1,k,i,j)$ and $\B''=(1,j,i,-k)$ respectively. 

  \item[(iii)] $|A|=3$.  If $\OO$ is axially symmetric, $k_i\le0$ for every $i=1,2,3$, $k_i=0$ for at most one index $i$ and $\sum_{i=1}^3 k_i=-1$, then $\F_{[3]}(\OO)=\sr(\OO)$. 
\end{itemize}
\end{proposition}
\begin{proof}
Points (i) and (iii) are immediate from Remark\ \ref{rem:FA}(c) and Theorem\ \ref{teo:slice_regularity}. 
We prove (ii). Let $J\in\s_\hh\cap\langle j,k\rangle\simeq S^1$ and consider the hypercomplex subspace $M_J=\langle 1,i,J\rangle$, with real coordinates $(x_0,x_1,\beta)$ w.r.t.\ the basis $\B_J=(1,i,J)$. 
The Cauchy-Riemann operator of $M_J$ is
\[
\dibar_{M_J}=\partial_{x_0}+i\partial_{x_1}+J\partial_\beta=\dibar_J,
\]
where the notation $\dibar_J$ is that used in \cite{GhiloniStoppatoJGP}. A function $f$ is $(1,3)$-regular on $\OO$ if and only if for every $J\in\s_\hh\cap\langle j,k\rangle$, the restriction of $f$ to $\OO\cap M_J$ is monogenic, namely, $\dibar_Jf(x_0+ix_1+J\beta)=0$ on $\OO\cap M_J$.  

Suppose that $f\in\F_{\{2,3\}}(\OO)$. 
Since $\underline S_{\{2,3\}}f=0$, from \eqref{eq:SA} we get, for $\x_A=x_2j+x_3k\ne0$,
\begin{equation}\label{eq:DD}
\DD_{\{2,3\}}f=-\x_{\{2,3\}}^{-1}\Eu_{\{2,3\}}f=\tfrac{x_2j+x_3k}{x_2^2+x_3^2}(x_2\partial_{x_2}f+x_3\partial_{x_3}f)=J\partial_\beta f,  
\end{equation}
where $J=\text{sign}(\beta)\tfrac{x_2j+x_3k}{\|x_2j+x_3k\|}$. 
Therefore $0=D_{\{2,3\}}f=(\partial_{x_0}+i\partial_{x_1}+\DD_{\{2,3\}})f=\dibar_J f$ for every $J$, and $f$ is $(1,3)$-regular. 

Let $\OO=\OO_E\subseteq\hh$ be $A$-circular, with $A=\{2,3\}$. We show that $f$ is a $(1,3)$-function \cite[\S4]{GhiloniStoppatoJGP}. Since $\mathscr S_{A}f=0$, 
from Proposition \ref{pro:SA} we get continuous functions 
$F_\emptyset,F_1:E'\times E''\subseteq\cc\times\rr\to\Aa$ such that for every $x=x_0+ix_1+\beta J\in\OO$, it holds
\[
f(x)=F_\emptyset(x_0+i\beta,x_1)+JF_1(x_0+i\beta,x_1),
\]
where $x_0+i\beta\in E'$.  
Therefore $f$ is induced by the $(1,3)$-stem function with continuous components $F_\emptyset$ and $F_1$ (we refer to \cite[\S4]{GhiloniStoppatoJGP} for notation). 

Conversely, if $f$ is a $(1,3)$-regular function, then in view of Remark \ref{rem:even-odd}, $f^\circ_{s,A}$ and ${(\x_A f)}^\circ_{s,A}$ 
satisfy the conditions of Proposition \ref{pro:SA}, that gives $\mathscr S_{\{2,3\}}f=0$. From this we get, as in \eqref{eq:DD}, that $D_{\{2,3\}}f(x_0+ix_1+J\beta)=\dibar_J f(x_0+ix_1+J\beta)=0$ for every $J\in\s_\hh\cap\langle j,k\rangle$, i.e., $D_{\{2,3\}}f=0$ on $\OO$. Therefore $f\in\F_{\{2,3\}}(\OO)$. 
\end{proof}


The space $\F_{\{2,3\}}(\OO)$ contains isomorphic copies of $\M(\OO\cap\langle 1,i\rangle,\hh)=Hol(\OO\cap\cc,\hh)$ and of $\sr(\OO\cap\langle1,j,k\rangle)$ (see Remark \ref{rem:FAsubsets}). 

\begin{example}\label{ex:linear}
A $\rr$-linear polynomial $f=a_0x_0+a_1x_1+a_2x_2+a_3x_3$, with $a_\ell\in\hh$, belongs to $\F_{\{2,3\}}(\hh)$ if and only if it is a right $\hh$-linear combination of $f_1=x_0+ix_1$ and $f_2=x_0+jx_2+kx_3$. 
Observe that $f_1\in Hol(\cc,\hh)$ and $f_2\in\sr(\langle1,j,k\rangle)$. Since
\begin{align*}
\DD_{\{2,3\}}&=j\left(\partial_{x_2}+\tfrac{k_2}{x_2}(1-r_2)\right)+k\left(\partial_{x_3}+\tfrac{k_3}{x_3}(1-r_3)\right),
\\
\underline S_{\{2,3\}}&=\x_{\{2,3\}}\DD_{\{2,3\}}+\Eu_{\{2,3\}},
\end{align*}
with $k_2+k_3=-1/2$, it holds
\begin{align*}
D_{\{2,3\}}f&=(\partial_{x_0}+i\partial_{x_1})f+\DD_{\{2,3\}}f=(a_0+ia_1)+ja_2(1+2k_2)+ka_3(1+2k_3),\\
\underline S_{\{2,3\}}f&=\x_{\{2,3\}}\DD_{\{2,3\}}f+\Eu_{\{2,3\}}f=\x_{\{2,3\}}\DD_{\{2,3\}}f+a_2x_2+a_3x_3.
\end{align*}
Therefore $D_{\{2,3\}}f=\underline S_{\{2,3\}}f=0$ if and only if $\DD_{\{2,3\}}f=-(a_0+ia_1)=:-a$ and $-(jx_2+kx_3)a+a_2x_2+a_3x_3=0$. This last condition implies that $a_2=ja$, $a_3=ka$. Therefore
\[
f=a_0x_0+a_1x_1+jax_2+kax_3
=(x_0+ix_1)(a_0-a)+(x_0+jx_2+kx_3)a
\]
Since $f_1,f_2$ are polynomials, the conditions $D_{\{2,3\}}f_i=\underline S_{\{2,3\}}f_i=0$ imply, in view of Theorems \ref{teo:poly_slice_regularity} and \ref{teo:slice_regularity}, that $\mathscr S_{\{2,3\}} f_i=0$  for $i=1,2$, i.e., $f_1,f_2\in\F_{\{2,3\}}(\hh)$. 
\end{example}

\subsubsection{Reduced quaternions}

\begin{proposition}
If $M=\hh_r=\langle 1,i,j\rangle$ and $\B=(1,i,j)$, then $n=2$ and there are two real vector spaces of $A$-Dunkl-regular functions  on an open set $\OO\subseteq\hh_r$, excluding the trivial case $\F_\emptyset(\OO)=\{0\}$:
\begin{itemize}
  \item[(i)] $|A|=1$. The spaces $\F_{\{1\}}(\OO)=\F_{\{2\}}(\OO)=\ker\dibar_{\B}$ are the space $\M(\OO)$ of monogenic functions on $\OO$ when $\hh_r$ is identified with the paravector space of $\rr_2\simeq\hh$. 

  \item[(ii)] $|A|=2$.  If $\OO$ is axially symmetric, $k_1,k_2\le0$ and $k_1+k_2=-1/2$, then $\F_{[2]}(\OO)=\sr(\OO)$ (the space of slice-monogenic functions on the Clifford algebra $\R_2\simeq\hh$). 
\end{itemize}
\end{proposition}
\begin{proof}
Points (i) and (ii) follow immediately from Remark\ \ref{rem:FA}(c) and Theorem\ \ref{teo:slice_regularity}. 
\end{proof}


\subsubsection{Moisil-Teodorescu regularity}

If one takes the hypercomplex subspace $M'=\langle 1,j,k\rangle$ of $\hh$, with basis $\B'=(1,-k,j)$, we have again $n=2$ and $\dibar_{\B'}=\partial_{x_0}-k\partial_{x_1}+j\partial_{x_2}$ is a multiple of the \emph{Moisil-Teodorescu operator} $\mathcal D_{MT}$ \cite{MT1931} in the variables $x_0,x_1,x_2$: namely,
\[
\mathcal D_{MT}=i\partial_{x_0}+j\partial_{x_1}+k\partial_{x_2}=i\dibar_{\B'}.
\]
It follows that $\F_{\{1\}}(\OO)=\F_{\{2\}}(\OO)=\ker\mathcal D_{MT}$ for $\OO\subseteq M'\simeq\rr^3$. Observe that the space $\F_{[2]}(\OO)=\sr(\OO)$ is an isomorphic copy of the space of slice-monogenic functions on the Clifford algebra $\R_2$.

\subsubsection{Paravectors in the Clifford algebras $\rr_n$}

\begin{proposition}\label{pro:FA_Clifford}
If $M=\rr^{n+1}\subseteq\rr_n$ is the paravector space with basis $\B=(1,e_1,\ldots,e_n)$, then we have the following spaces of $A$-Dunkl-regular functions on an open set $\OO\subseteq M$, excluding the trivial case $\F_\emptyset(\OO)=\{0\}$:
\begin{itemize}
  \item[(i)] $|A|=1$. The spaces $\F_{\{1\}}(\OO)=\cdots=\F_{\{n\}}(\OO)=\ker\dibar_\B$ are the space $\M(\OO)$ of monogenic functions on $\OO$.

  \item[(ii)] $1<|A|<n$. Let $k_i\le0$ for all $i\in A$ with $k_i=0$ for at most one $i\in A$ and $\sum_{i\in A}k_i=(1-|A|)/2$. Let $\OO$ be $A$-circular.  
  If $A=\{n-\ell+1,\ldots,n\}$, with $\ell=|A|$, 
  then the space $\F_{A}(\OO)$ is the set of strongly $T$-regular functions on $\OO$, with $T=(n-|A|,n)$. 
  (see \cite[\S3]{GhiloniStoppatoJGP} and \cite{GhiloniStoppato_arXiv24}). Equivalently, $f\in\F_A(\OO)$ if and only if $f$ is a generalized partial-slice monogenic function of type $(n-|A|,|A|)$, a concept introduced in \cite{Sabadini_Xu_TAMS}.
  \\
  If $|A|=\ell$ but $A\not=A_\ell:=\{n-\ell+1,\ldots,n\}$, then the space $\F_A(\OO)$ is isomorphic to $\F_{A_\ell}(\OO)$ through 
  a  variables permutation of $M=\rr^{n+1}$ that sends a hypercomplex basis of $M_A$ onto the hypercomplex basis $(1,e_{n-\ell+1},\ldots,e_n)$ of $M_{A_\ell}$.

  \item[(iii)] $|A|=n$.  If $\OO$ is axially symmetric, $k_i\le0$ for every $i$, $k_i=0$ for at most one index $i$ and $\sum_{i=1}^n k_i=(1-n)/2$, then $\F_{[n]}(\OO)=\sr(\OO)=\mathcal{SM}(\OO)$, the space of slice-monogenic functions on $\OO$.
\end{itemize}
The number of distinct spaces $\F_A(\OO)$, excluding $\F_\emptyset(\OO)=\{0\}$, is $2^n-n$. Of these, only $n$ are  essentially inequivalent. 
\end{proposition}
\begin{proof}
Points (i) and (iii) are immediate from Remark\ \ref{rem:FA}(c) and Theorem\ \ref{teo:slice_regularity}. 

We prove the first statement of (ii), adapting the proof given in the quaternionic case (Proposition \ref{pro:Dunkl_H}) to Clifford algebras.   If $A=\{n-\ell+1,\ldots,n\}$, let $M_A:=\langle 1,e_{n-\ell+1},\ldots,e_n\rangle$. Fix an imaginary unit $J$ in the $(\ell-1)$-dimensional unit sphere $\s_{M_A}=\s_{\rr_n}\cap M_A$ of $M_A\cap\ker(t)$. 
Consider the hypercomplex subspace $M_J=\langle 1,e_1,\ldots,e_{n-\ell},J\rangle$, with real coordinates $(x_0,x_1,\ldots,x_{n-\ell},\beta)$ w.r.t.\ the basis $\B_J=(1,e_1,\ldots,e_{n-\ell},J)$. 
The Cauchy-Riemann operator of $M_J$ is
\[
\dibar_{M_J}=\partial_{x_0}+\sum_{i=1}^{n-\ell}e_i\partial_{x_i}+J\partial_\beta=\dibar_J,
\]
where the notation $\dibar_J$ is that used in \cite[\S3]{GhiloniStoppatoJGP}. A function $f$ is strongly $T$-regular on $\OO$, with $T=(n-|A|,n)$, if and only if for every $J\in\s_{\rr_n}\cap M_A$, the restriction of $f$ to $\OO\cap M_J$ is monogenic, namely, $\dibar_Jf(x_0+e_1x_1+\cdots+e_{n-\ell}x_{n-\ell}+J\beta)=0$ on $\OO\cap M_J$.  

Given $f\in\F_A(\OO)$, it holds $\underline S_Af=0$. From \eqref{eq:SA} we get, for $\x_A=e_{n-\ell+1}x_{n-\ell+1}+\cdots+e_{n}x_{n}\ne0$, that
\begin{equation}\label{eq:DD_Clifford}
\DD_Af=-\x_A^{-1}\Eu_A f=\tfrac{\x_A}{\|\x_A\|}\sum_{j=n-\ell+1}^{n}\tfrac{x_j}{\|\x_A\|}\partial_{x_j}f=J\partial_\beta f,  
\end{equation}
where $J=\text{sign}(\beta)\tfrac{\x_A}{\|\x_A\|}\in\s_{\rr_n}\cap M_A$. 
Therefore $0=D_{A}f=(\partial_{x_0}+\sum_{i=1}^{n-\ell}e_i\partial_{x_i}+\DD_{A})f=\dibar_J f$ for every $J$, and $f$ is strongly $T$-regular. 

We show that $f$ is a $T$-function \cite[\S3]{GhiloniStoppatoJGP}. 
Since $\mathscr S_{A}f=0$, from Proposition \ref{pro:SA} we get continuous functions 
$F_\emptyset^A,F_1^A:E'\times E''\subseteq\cc\times\rr^{n-\ell}\to\Aa$ such that for every $x=x_0+\sum_{i=1}^{n-\ell}e_ix_i+\beta J\in\OO$, it holds
\[
f(x)=F_\emptyset^A(x_0+i\beta,y)+JF_1^A(x_0+i\beta,y),
\]
where $x_0+i\beta\in E'$ and $y=(x_1,\ldots, x_{n-\ell})\in E''$. Therefore $f$ is induced by the $T$-stem function with continuous components $F_\emptyset^A$ and $F_1^A$.  


Conversely, if $f$ is a strongly $T$-regular function, then $f^\circ_{s,A}$ and ${(\x_A f)}^\circ_{s,A}$ satisfy the conditions of Proposition \ref{pro:SA}, that gives $\mathscr S_{A}f=0$, from which we get, as in \eqref{eq:DD_Clifford}, that $D_{A}f(x_0+e_1x_1+\cdots+e_{n-\ell}x_{n-\ell}+J\beta)=\dibar_J f(x_0+e_1x_1+\cdots+e_{n-\ell}x_{n-\ell}+J\beta)=0$ for every $J\in\s_{\rr_n}\cap M_A$, i.e., $D_{A}f=0$ on $\OO$ and $f\in\F_A(\OO)$. 

Now we prove the last statement of (ii). Let $A=\{i_1,\ldots,i_\ell\}$ and $A^c=[n]\setminus A=\{j_1,\ldots,j_{n-\ell}\}$, with increasing sequences $\{i_k\}$, $\{j_k\}$. 
Consider the isomorphism $T$ of $M=\rr^{n+1}$ such that $T(1)=1$, $T(e_{i_j})=e_{n-\ell+j}$ for $j=1,\ldots,\ell$ and $T(e_{j_k})=e_{k}$ for $k=1,\ldots,n-\ell$. Then the Dunkl-Cauchy-Riemann operator $D_{A}$ w.r.t.\ the basis $\B=(1,e_1,\ldots,e_n)$ becomes the Dunkl-Cauchy-Riemann operator $D_{A_\ell}$ w.r.t.\ the basis $T^{-1}(\B)$. 
\end{proof}

\begin{example}\label{ex:axiallymonogenic}
An example of function space $\F_{A,B}(\OO)$, with $A\not=B$ (see Definition \ref{def:FAB}), is that of \emph{axially monogenic functions} on Clifford algebras.  If $\OO\subseteq\rr_n$ is axially symmetric, and if the Dunkl multiplicities $k_i$ in $\DD_{[n]}=\DD_\B$ are non-positive, with at most one zero multiplicity, and $\sum_{i=1}^nk_i=(1-n)/2$, then $\F_{\emptyset,[n]}(\OO)=\M(\OO)\cap\SL(\OO)$ is the space $\AM(\OO)$ of axially monogenic functions on $\OO$. 
\end{example}

More examples involving the Clifford algebra $\R_n$ will be provided in the forthcoming Subsection \ref{sub:Slice_Dunkl-regular_functions}.

\subsubsection{Octonions}
The algebra isomorphism $\oo\simeq\hh\oplus\hh$ provided by the Cayley-Dickson construction induces multiplication 
$(q_1,q_1')(q_2,q_2')=(q_1q_2-q_2'\overline{q_1'},\overline{q_1}q_2'+q_2q_1')$ and conjugation $(q,q')^c=(\overline{q},-q')$. 
As in \cite{DentoniSce}, let $\B$ be the orthonormal basis of $\oo\simeq\R^8$ formed by elements $v_h=(e_h,0)$ for $h=0,1,2,3$ and $v_h=(0,e_{h-4})$ for $h=4,5,6,7$,  where $(e_0=1,e_1,e_2,e_3)$ is the standard basis of $\hh$.  The operator $\dB$ is the octonionic Cauchy-Riemann operator (firstly introduced by Dentoni and Sce in \cite{DentoniSce} as the \emph{Fueter-Moisil operator}) 
\[
\dibar_\oo=\partial_{x_0}+v_1\partial_{x_1}+\cdots +v_7\partial_{x_7},
\]
where $x_0,\ldots,x_7$ are the real coordinates w.r.t.\ $\B$. 
Octonionic slice-regular functions were introduced in \cite{GeStRocky}.

\begin{proposition}\label{pro:Dunkl_O}
If $M=\oo$, then $n=7$ and there are 121 (seven essentially inequivalent) 
real vector spaces of $A$-Dunkl-regular functions  on an open set $\OO\subseteq\hh$, excluding the trivial case $\F_\emptyset(\OO)=\{0\}$:
\begin{itemize}
  \item[(i)] $|A|=1$. The spaces $\F_{\{1\}}(\OO)=\cdots=\F_{\{7\}}(\OO)=\ker\dibar_{\oo}$ are the space of octonionic monogenic functions on $\OO$.

  \item[(ii)] 
  $1<|A|<7$. Let $k_i\le0$ for all $i\in A$ with at most one $i\in A$ such that $k_i=0$ and $\sum_{i\in A}k_i=(1-|A|)/2$. Let $\OO$ be $A$-circular.  
  If $A=\{8-\ell,\ldots,7\}$, with $\ell=|A|$, 
  then the space $\F_{A}(\OO)$ is the set of strongly $T$-regular functions (\cite[\S3]{GhiloniStoppatoJGP}, \cite{GhiloniStoppato_arXiv24} and \cite{Sabadini_Xu_Octo_TAMS}) of class $C^1$ on $\OO$, with $T=(7-|A|,7)$. 
  \\
  If $|A|=\ell$ but $A\not=A_\ell:=\{8-\ell,\ldots,7\}$, the space $\F_A(\OO)$ is isomorphic to $\F_{A_\ell}(\OO)$ through 
  a  variables permutation of $\oo=\rr^{8}$ that sends a hypercomplex basis of $M_A$ onto the hypercomplex basis $(1,v_{8-\ell},\ldots,v_7)$ of $M_{A_\ell}$.  

  \item[(iii)] $|A|=7$.  If $\OO$ is axially symmetric, $k_i\le0$ for every $i=1,2,\ldots,7$, $k_i=0$ for at most one index $i$ and $\sum_{i=1}^7 k_i=-3$, then $\F_{[7]}(\OO)$ is the space $\sr(\OO)$ of octonionic slice-regular functions.
\end{itemize}
\end{proposition}
\begin{proof}
The proof follows the same lines of the proofs of Propositions \ref{pro:Dunkl_H} and \ref{pro:FA_Clifford}. 
\end{proof}

\subsection{Slice Dunkl-regular functions}
\label{sub:Slice_Dunkl-regular_functions}

Example \ref{ex:axiallymonogenic} can be generalized to every function space $\F_A(\OO)$.  If $\OO\subseteq\Aa$ is axially symmetric, one can consider the space 
\[
\F_{A,[n]}(\OO)=\{f\in C^1(\OO,\Aa)\;|\; f\in\ker D_A\cap\ker \mathscr S_{[n]}\}.
\]
It holds $\F_{A,[n]}(\OO)=\F_A(\OO)\cap\SL(\OO)$, i.e., $\F_{A,[n]}(\OO)$ is the space of Dunkl-regular functions in $\F_A(\OO)$ that are also slice functions on $\OO$. The inclusion $\F_{A,[n]}(\OO)\supseteq\F_A(\OO)\cap\SL(\OO)$ is immediate from definitions, since slice functions are in the kernel of $\mathscr S_{[n]}$. The reverse inclusion follows from Proposition \ref{pro:kerS}: if $f\in\F_{A,[n]}(\OO)$, then $f\in\ker\mathscr S_A$ and therefore $f\in\F_A(\OO)$.   

\begin{example}
If $\Aa=\oo$, it can be shown
that the space $\F_{\{3,4,5,6,7\},[7]}(\OO)$ is the space of \emph{slice Fueter-regular} functions on $\OO$, a concept introduced and investigated in \cite{JinRenSabadini2020,SliceFueterRegular}. In particular, every slice Fueter-regular function is Dunkl monogenic, with multiplicities ${\bf k}=(0,0,k_3,\ldots,k_7)$ such that $\sum_{i=3}^7 k_i=-2$. 
\end{example}

\begin{example}
If $\Aa=\hh$, the space $\F_{\{2,3\},[3]}(\OO)$ is the space of slice functions on $\OO$ that are $(1,3)$-regular on $\OO$ (see Proposition \ref{pro:Dunkl_H}). For example, the $\rr$-linear function $f=f_1+f_2=2x_0+x_1i+x_2j+x_3k$, with $f_1,f_2$ as in Example \ref{ex:linear}, is a slice function in the space $\F_{\{2,3\}}(\hh)$.
\end{example}


\subsection{Dunkl-regular function spaces defined by partitions}
\label{sub:Dunkl-regular_function_spaces_defined_by_partitions}

Let $\B=(1,v_1,\ldots,v_n)$ be a hypercomplex basis of the hypercomplex subspace $M$ of $\Aa$, with associated coordinates $x_0,x_1,\ldots,x_n$. 
Given a partition  $\PP=\{A_1,\ldots,A_\ell\}$ of the set $[n]$, let ${\bf k}=(k_1,\ldots,k_n)$ be a sequence of non-positive Dunkl multiplicities such that, for every $j=1,\ldots,\ell$, it holds $2\sum_{i\in A_j}k_i=1-|A_j|$ and $k_i=0$ for at most one index $i\in A_j$. 
We shall call \emph{$\PP$-admissible} a sequence $\bf k$ of multiplicities satisfying these conditions. 

\begin{definition}\label{def:P-Dunkl-regular}
Let $\OO\subseteq M$ be open and $\zz_2^n$-invariant. Let $\PP=\{A_1,\ldots,A_\ell\}$ be a partition of $[n]$. Fix a sequence ${\bf k}=(k_1,\ldots,k_n)$ of $\PP$-admissible Dunkl multiplicities. We call \emph{$\PP$-Dunkl-regular 
functions} on $\OO$ the elements of the real vector space (a right $\Aa$-module if $\Aa$ is associative)
\[
\F_{\PP}(\OO):=\{f\in C^1(\OO,\Aa)\;|\; f\in\ker D_\PP\cap\ker \mathscr S_\PP\},
\] 
where $D_\PP$ is the $\zz_2^n$ Dunkl-Cauchy-Riemann operator 
$D_\PP=\partial_{x_0}+\sum_{j=1}^\ell\DD_{A_j}$ and $\mathscr S_\PP=(\mathscr S_{A_1},\ldots,\mathscr S_{A_\ell})$. We also say that a function is \emph{Dunkl-regular} if it is $\PP$-Dunkl-regular for some $\PP$. 
\end{definition}

We can rewrite the operator $D_\PP$ as
\begin{equation}\label{eq:DP}
  D_\PP=\dM+\sum_{i=1}^n \frac{k_iv_i}{x_i}(1-r_i),
\end{equation}

\begin{remark}
Two different partitions $\PP$, $\PP'$ can define the same operator $D_\PP$ but not the same space $\F_\PP(\OO)$: for example, if $n=4$, $\PP=\{\{1\},\{2,3,4\}\}$, $\PP'=\{\{1,2\},\{3,4\}\}$, the choice of multiplicities
\[
{\bf k}=(k_1,k_2,k_3,k_4)=(0,-1/2,-1/4,-1/4),
\]
is admissible for both $\PP$ and $\PP'$. Equation \eqref{eq:DP} yields $D_\PP=D_{\PP'}$. Observe however that $\mathscr S_\PP$ and $\mathscr S_{\PP'}$ are different: for example, the function $\x_{\{2,3,4\}}$ belongs to $\ker\mathscr S_\PP$ but $\x_{\{2,3,4\}}\not\in\ker\mathscr S_{\PP'}$ since $\underline S_{\{1,2\}}(\x_{\{2,3,4\}})\ne0$.  We shall prove in the forthcoming Theorem \ref{teo:classification1} that $\F_\PP(\OO)$ and $\F_{\PP'}(\OO)$ are always different if $\PP\ne\PP'$. 
\end{remark}

\begin{remark}\label{rem:FP}
(a)\quad Since $\F_{\PP}(\OO)\subseteq\ker D_\PP$, every Dunkl-regular function is, in particular, Dunkl monogenic on $\OO$ whence Dunkl harmonic.

(b)\quad
If $A\subseteq[n]$ is not empty, $A^c=\{j_1,\ldots,j_{n-\ell}\}$ and $\PP=\{A,\{j_1\},\ldots,\{j_{n-\ell}\}\}$, then $D_\PP=D_A$ and $\mathscr S_\PP=(\mathscr S_A,0,\ldots,0)$.
It follows that 
$\F_\PP(\OO)$ coincides with the space $\F_A(\OO)$ of $A$-Dunkl-regular functions of Definition \ref{def:A-Dunkl-regular}.

\begin{itemize}
  \item[(i)]
  $\F_{\{\{1\},\ldots,\{n\}\}}(\OO)=\M(\OO)$ on any open subset $\OO$ of $M$.
  \item[(ii)]
  $\F_{\{\{1,\ldots,n\}\}}(\OO)=\F_{[n]}(\OO)=\sr(\OO)$ on any axially symmetric open subset $\OO$ of $M$.
  \item[(iii)]
  On the Clifford algebra $\rr_n$, the function space of generalized partial-slice monogenic functions of type $(p,n-p)$ (see \cite{Sabadini_Xu_TAMS}) coincides with the space $\F_\PP(\OO)$, with $\PP=\{\{1\},\ldots,\{p\},\{p+1,\ldots,n\}\}$. This follows from Proposition \ref{pro:FA_Clifford} and point (b) above. 
  \item[(iv)]
  Again from Proposition \ref{pro:FA_Clifford} and (b), the space of strongly $T$-regular functions on $\OO\subseteq\rr_n$, with $T=(t_0,t_1=n)$ (see \cite{GhiloniStoppatoJGP}), is the space $\F_\PP(\OO)$, with $\PP=\{\{1\},\ldots,\{t_0\},\{t_0+1,\ldots,n\}\}$. 
\end{itemize}

(c)\quad
Let $\PP=\{A_1,\ldots,A_\ell\}$. The inclusion $M_{A_i}\hookrightarrow M$ induces the vector space inclusion of $\sr(\OO\cap M_{A_i})$ into $\F_\PP(\OO)$. In particular, the linear function $x_{A_i}=x_0+\x_{A_i}=x_0+\sum_{j\in A_i}v_jx_j$ belongs to $\F_\PP(\OO)$ for every $i=1,\ldots,\ell$.
\end{remark}

Let $\PP=\{A_1,\ldots,A_\ell\}$ be a partition of the set $[n]$. Then $\PP$ defines a partition of the number $n$ as $n=|A_1|+\cdots+|A_\ell|$. 
The \emph{number of partitions} of the set $[n]=\{1,\ldots,n\}$ is the \emph{Bell number} $B_n$, while the number of partitions of the integer $n$ is the \emph{partition number} $p(n)$.

\begin{theorem}\label{teo:classification1}
Let $n=\dim M-1$ and let $\OO\subseteq M$ be open and $\zz_2^n$-invariant. 
The function space $\F_\PP(\OO)$ does not depend on the Dunkl multiplicities of the operators $D_\PP$ and $\mathscr S_\PP$, provided that they are $\PP$-admissible. Given two different partitions $\PP$ and $\PP'$ of the set $[n]$, the function spaces $\F_\PP(\OO)$ and $\F_{\PP'}(\OO)$ are distinct. 
\end{theorem}

\begin{proof}
Let $\PP=\{A_1,\ldots,A_\ell\}$ and let $D_\PP$, $\mathscr S_\PP$ and $D'_{\PP}$, $\mathscr S'_{\PP}$ operators with $\PP$-admissible Dunkl multiplicities $k_i$ and $k'_i$ respectively. The characterization of $\ker\mathscr S_A$ given by Proposition \ref{pro:SA} does not depend on the multiplicities whence $f\in\ker\mathscr S_\PP$ if and only if $f\in\ker\mathscr S'_\PP$. 
If $f\in\ker D_\PP\cap\ker\mathscr S_\PP$, then $f\in\ker \underline S_{A_i}$ for every $i$ and, from \eqref{eq:SA}, $\DD_{A_i}f=-\x_{A_i}^{-1}\Eu_{A_i}f$. Therefore also $D_\PP f=(\partial_{x_0}-\sum_{i=1}^\ell\x_{A_i}^{-1}\Eu_{A_i})f$ does not depend on the $k_i$'s, and $f\in\ker D'_\PP\cap\ker\mathscr S'_\PP$. 

Let $\PP=\{A_1,\ldots,A_\ell\}$, $\PP'=\{B_1,\ldots,B_m\}$ different partitions of the set $[n]$ and assume that $A_i\not\in\PP'$. Let $j$ be such that $A_i\cap B_j\ne\emptyset$. If $A_i\setminus B_j\ne\emptyset$, then $(x_{B_j})_{|M_{A_i}}=x_{A_i\cap B_j}$ is not a slice function on $M_{A_i}$, since
\[\textstyle
\underline S_{A_i}(x_{B_j})=\x_{A_i}\DD_{A_i}(x_{B_j})+\Eu_{A_i}(x_{B_j})=\x_{A_i}\sum_{h\in A_i\cap B_j}(-1-2k_h)+\x_{A_i\cap B_j}
\ne0.
\]
Then $x_{B_j}\in\F_{\PP'}(\OO)\setminus\F_\PP(\OO)$. If instead $B_j\setminus A_i\ne\emptyset$, then $x_{A_i}\in\F_{\PP}(\OO)\setminus\F_{\PP'}(\OO)$. 
\end{proof}

\begin{corollary}\label{cor:Kreg}
Let $\PP=\{A_1,\ldots,A_\ell\}$. The Dunkl multiplicities $k_i$ of the operators $D_\PP$, $\mathscr S_\PP$ defining the space $\F_\PP(\OO)$ can be taken as $k_i=-\frac12+\frac1{2|A_j|}\in(-\frac12,0]$ for every $i\in A_j$ and $j=1,\ldots,\ell$. \hfill\qed
\end{corollary}

\begin{remark}\label{rem:Kreg}
In view of the Corollary, the Dunkl multiplicities of the operators $D_\PP$, $\mathscr S_\PP$ in $\F_\PP(\OO)$, although non-positive, can be chosen in the \emph{regular set} $K^{reg}$ (see e.g.\ \cite[\S2.4]{Rosler}) for the reflection group $\zz_2^n$.  
\end{remark}

In general, the function space $\F_\PP(\OO)$ depends on the choice of the hypercomplex basis $\B$ of $M$. If we want to highlight this dependence, we will write $\F_\PP^\B(\OO)$ instead of $\F_\PP(\OO)$.

Let $\sigma$ be a permutation of the set $[n]$. Given a partition $\PP=\{A_1,\ldots,A_\ell\}$ of $[n]$, let $\PP_\sigma=\{B_1,\ldots,B_\ell\}$ be the partition defined by $B_i=\sigma^{-1}(A_i)$, $i=1,\ldots,\ell$. 
Observe that a sequence ${\bf k}=(k_1,\ldots,k_n)$ of Dunkl multiplicities is $\PP$-admissible if and only if the sequence ${\bf k}_\sigma=(k_{\sigma(1)},\ldots,k_{\sigma(n)})$ is $\PP_\sigma$-admissible. 

Let $\B_\sigma=(1,v'_1,\ldots,v'_{n})$ be the basis obtained by $\B$ permuting its elements: $v'_i=v_{\sigma(i)}$ for every $i=1\ldots,n$. 

We shall call \emph{equivalent} two function spaces $\F^\B_\PP(\OO)$ and $\F^{\B'}_{\PP'}(\OO)$  if $\B'=\B_\sigma$ and $\PP'=\PP_\sigma$ for some permutation $\sigma$.

\begin{theorem}\label{teo:classification2}
Let $n=\dim M-1$ and let $\OO\subseteq M$ be open and $\zz_2^n$-invariant. Let $\B$ be a fixed hypercomplex basis.
If the partitions $\PP$ and $\PP'$ of the set $[n]$ define the same partition of the number $n$, then there exists a hypercomplex basis $\B'$ of $M$ such that the spaces $\F_\PP^\B(\OO)$ and $\F_{\PP'}^{\B'}(\OO)$ are equivalent. 
Such function spaces are isomorphic as real vector spaces.  
\end{theorem}

\begin{proof}

If $\PP=\{A_1,\ldots,A_\ell\}$ and $\PP'=\{B_1,\ldots,B_m\}$ define the same partition of $n$, then $\ell=m$ and, up to reordering, we can assume that $|A_i|=|B_i|$ for every $i=1,\ldots,\ell$. Let $\sigma$ be a permutation of $[n]$ such that $B_i=\sigma(A_i)$ for every $i$. Then $\PP'=\PP_\sigma$. 

Let $\B'=\B_\sigma=\{1,v_1',\ldots,v_n'\}$, with coordinates $y_0,\ldots,y_n$. 
Let $L:\R^d \to \Aa$ be the real vector isomorphism defined by a basis $\B_{\Aa}$ of $\Aa$ that completes $\B=\{1,v_1,\ldots,v_n\}$, and $P\in O(n)$ the permutation matrix 
such that $v_j=\sum_{i=1}^np_{ij}v'_i$ for every $j=1,\ldots,n$. If $L'$ is the isomorphism obtained from $\B'$, then it holds $(x_0,\ldots,x_{n})=(y_0,\ldots,y_{n})\widehat P$, where $\widehat P$ is the permutation matrix
\[\widehat P=\begin{bmatrix}
1&0\\
0&P
\end{bmatrix}\in O(n+1).
\]
The operators $D_\PP$ and $\mathscr S_\PP$  w.r.t.\ the basis $\B$ 
with  sequence ${\bf k}$ coincide, respectively, with the operators $D_{\PP_\sigma}$ and $\mathscr S_{\PP_\sigma}$ w.r.t.\ the basis $\B'=\B_\sigma$ 
with  sequence ${\bf k}_\sigma$. 
Therefore a function $f$ belongs to the space $\F^\B_{\PP}(\OO)$ if and only if the function $f_\sigma$ defined by the relation
\[
f_\sigma\circ L'_{|U'} = f\circ L_{|U}\circ R_{\widehat P}
\]
belongs to $\F_{\PP'}^{\B'}(\OO)$. Here $U=(L_{|\rr^{n+1}})^{-1}(\OO)$ and $U'=(L'_{|\rr^{n+1}})^{-1}(\OO)$ are subsets of $\rr^{n+1}$ and $R_{\widehat P}$ denotes the right multiplication by $\widehat P$. The correspondence $\F_\PP^\B(\OO)\ni f \leftrightarrow f_\sigma\in\F_{\PP'}^{\B'}(\OO)$ defines the isomorphism between the two real vector spaces.
\end{proof}

\begin{corollary}\label{cor:equivalence}
Two spaces $\F_\PP^\B(\OO)$ and $\F_{\PP'}^{\B'}(\OO)$ of Dunkl-regular functions are equivalent if and only if $\PP$ and $\PP'$ define the same partition of $n=\dim M-1$. Therefore the number of inequivalent functions spaces $\F_\PP(\OO)$, $\OO\subseteq M$, is equal to the partition number $p(n)$. 
In particular, there are $2^n-n$ distinct function spaces of the form $\F_A(\OO)\ne\{0\}$, with $A$ a non-empty subset of $[n]$ (see Definition \ref{def:A-Dunkl-regular}). Exactly $n$ of these spaces are inequivalent, one for each cardinality $1\le|A|\le n$.   
\end{corollary}

Observe that the spaces $\M(\OO)$ and $\sr(\OO)$, associated to the partitions $\{\{1\},\ldots,\{n\}\}$  and $\{\{1,\ldots,n\}\}$  respectively, are the unique functions spaces $\F_\PP^\B(\OO)$ that are not equivalent to any other space $\F_{\PP'}^{\B'}(\OO)$. 

\begin{table}[h]
\begin{center}
\begin{tabular}{c|c|c|c|c|c|c|c|c|c|c|c}
\hline &&&&&&&&&&\\[-10pt]
$n=\dim(M)-1$ &$1$&$2$&$3$&$3$&$3$&$4$&$5$&$6$&7&$8$&\dots \\
\hline &&&&&&&&&&\\[-10pt]
$p(n)$ & $1$ & 2&3&3&3&5&7&11&15&22&\dots\\
\hline &&&&&&&&&&\\[-10pt]
$B_n$ & $1$ & $2$ &$5$&$5$&$5$&$15$&$52$&$203$&$877$&4140 &\dots\\
\hline &&&&&&&&&&\\[-10pt]
$2^n-n$ &$1$&$2$&$5$&$5$&$5$&$12$& 27& 58& 121& 248&\dots \\
\hline &&&&&&&&&&\\[-10pt]
$M$&$\cc$ &$\hh_r$ & $\hh$ &$\hh$&$\rr^4$&$\rr^5$ & $\rr^6$&$\rr^7$ & $\oo$&$\rr^9$&\dots\\
\hline &&&&&&&&&&\\[-10pt]
$\Aa$&$\cc$ &$\hh$ & $\hh$ &$\oo$&$\rr_3$&$\rr_4$ & $\rr_5$&$\rr_6$ & $\oo$&$\rr_8$&\dots\\
\hline
\end{tabular}
\vskip5pt
\caption{Partition numbers and Bell numbers for some hypercomplex subspaces $M$}
\label{table}
\end{center}
\end{table}


\subsection{Dunkl-regularity and $T$-regular functions}
\label{sub:Dunkl-regularity_and_regular_$T$-functions}

Remark \ref{rem:FP}(iv) can be generalized to include in the theory of Dunkl-regular functions all the classes of strongly $T$-regular functions on a hypercomplex subspace $M$ of $\Aa$ investigated 
in \cite{GhiloniStoppatoJGP,GhiloniStoppato_arXiv24}. We refer to that literature for definitions.

\begin{definition}
Let $\PP=\{A_1,\ldots,A_\ell\}$. A subset $\OO$ of $M$ is called $\PP$-circular if it is $A_j$-circular for every $j=1,\ldots,\ell$. Equivalently, $\OO$ is $\PP$-circular if and only if there exists $D\subseteq\rr^{\ell+1}$ such that $\OO=\OO_D$, where 
\[\textstyle
\OO_D:=\left\{x_0+\sum_{i=1}^\ell\beta_i J_i\;|\; (x_0,\beta_1,\ldots,\beta_\ell)\in D, J_i\in \s_{M_{A_i}}\text{\ for }i=1,\ldots\ell\right\}.
\]
\end{definition}



\begin{theorem}\label{teo:Tregular}
Given a list of steps, namely a sequence $T=(t_0,\ldots,t_\tau)$ of integers with $\tau>0$ and $0\le t_0<t_1<\cdots <t_\tau=n$, let $\PP$ be the partition of $[n]$ defined as:
\begin{align*}
\PP&=\{\{1,\ldots,t_1\},\{t_1+1,\ldots,t_2\},\ldots,\{t_{\tau-1}+1,\ldots,t_\tau\}\}\text{\quad if $t_0=0$, and}\\
\PP&=\{\{t_0+1,\ldots,t_1\},\{t_1+1,\ldots,t_2\},\ldots,\{t_{\tau-1}+1,\ldots,t_\tau\},\{1\},\ldots,\{t_0\}\}\text{\quad if $t_0>0$.}
\end{align*}
Let $\OO$ be a $\PP$-circular open subset of $M$. Then $\OO$ is a $T$-symmetric set \cite[Definition 4.20]{GhiloniStoppato_arXiv24} and $\F_\PP(\OO)$  is the set of strongly $T$-regular functions on $\OO$.
\end{theorem}

From Theorems \ref{teo:Tregular} and \ref{teo:classification2} we get the following:

\begin{corollary} $\F_\PP(\OO)$ is always isomorphic to the space of strongly $T$-regular functions on $\Omega$ for some list of steps $T$.
\end{corollary}

In view of the latter result, one could apply Corollary \ref{cor:equivalence} to study the  
classification of spaces of $T$-regular functions up to bijections, a problem proposed in \cite[\S5.3]{GhiloniStoppato_arXiv24}.

Before proving the theorem, we investigate the kernel of the operator $\mathscr S_\PP$. We start by considering the intermediate case when $f$ belongs to the kernel of two operators $\mathscr S_{A_1}$, $\mathscr S_{A_2}$, with $A_1, A_2 \subseteq [n]$ that do not intersect.
In the following, if $K=\{k_1,\ldots,k_p\}\subseteq[n]$, we write indifferently $F_K$ or $F_{k_1\cdots k_p}$ to denote a function $F_K$ associated to $K$. We recall the notation
$\bar \beta^h=(\beta_1,\ldots,-\beta_h\ldots,\beta_\ell)$ to denote the reflection of $\beta\in\rr^\ell$ w.r.t.\ the $h$-th variable.

\begin{proposition}\label{pro:kerS2}
Let $A_1,A_2$ be non-empty subsets of $[n]$, with $A_1\cap A_2=\emptyset$. Let $(A_1\cup A_2)^c=\{j_1,\ldots,j_{n-\ell}\}$ with $j_1<\cdots <j_{n-\ell}$. 
Let $\OO$ be an open subset of $M$ which is $A_1$-circular and $A_2$-circular.   
If $f\in\ker\mathscr S_{A_1}\cap \ker\mathscr S_{A_2}$,  then there exist continuous functions 
$F_\emptyset,F_{1},F_{2},F_{12}:E'\times E''\subseteq\rr^3\times\rr^{n-\ell}\to\Aa$ such that for every $x\in\OO$ it holds
\[
f(x)=F_{\emptyset}(x_0,\beta,y)+J_1F_{1}(x_0,\beta,y)+J_2F_{2}(x_0,\beta,y)+J_1(J_2F_{12}(x_0,\beta,y)). 
\]
Here $(x_0,\beta)=(x_0,\beta_1,\beta_2)\in E'$, $y=(x_{j_1},\ldots,x_{j_{n-\ell}})\in\rr^{n-\ell}\in E''$ and $x=x_0+\sum_{i=1}^{n-\ell}v_{j_i}x_{j_i}+\beta_1 J_1+\beta_2 J_2\in\OO$. For $i=1,2$, the unit $J_i\in \s_M$ is defined as $J_i:=\beta_i^{-1}\x_{A_i}$ if $|\beta_i|=\|\x_{A_i}\|\ne0$, and $J_i$ equal to any element of $\s_{M_{A_i}}$ if $\beta_i=0$, $\x_{A_i}=0$.
The function with components $\{F_\emptyset,F_{1},F_{2},F_{12}\}$ is a $T$-stem function w.r.t.\ $\beta=(\beta_1,\beta_2)$, namely it satisfies $F_K(x_0,\bar \beta^h,x')=(-1)^{|K\cap\{h\}|}F_K(x_0,\beta,x')$ for $h=1,2$, $K\subseteq\{1,2\}$.  
\end{proposition}

\begin{proof}
Since $f\in\ker\mathscr S_{A_1}$, Proposition \ref{pro:SA} and Remark \ref{rem:even-odd} give
\[
f=f^\circ_{s,A_1}+\x_{A_1}^{-1}(\x_{A_1}f)^\circ_{s,A_1}. 
\]
The functions $f^\circ_{s,A_1}$ and $(\x_{A_1}f)^\circ_{s,A_1}$ belong to $\ker\mathscr S_{A_2}$. This can be justified directly using the definitions and the commutation properties of Proposition \ref{pro:Commutation relations}, which follow from condition $A_1\cap A_2=\emptyset$. 

We can apply again  Proposition \ref{pro:SA} to $f^\circ_{s,A_1}$ and $(\x_{A_1}f)^\circ_{s,A_1}$ and get
\begin{align*}
f&=(f^\circ_{s,A_1})^\circ_{s,A_2}+\x_{A_2}^{-1}(\x_{A_2}f^\circ_{s,A_1})^\circ_{s,A_2} 
+\x_{A_1}^{-1}
\left(((\x_{A_1}f)^\circ_{s,A_1})^\circ_{s,A_2}+\x_{A_2}^{-1}\left((\x_{A_2}(\x_{A_1}f)^\circ_{s,A_1})^\circ_{s,A_2}\right) \right)\\
&=(f^\circ_{s,A_1})^\circ_{s,A_2}+\x_{A_1}^{-1}((\x_{A_1}f)^\circ_{s,A_1})^\circ_{s,A_2}
+\x_{A_2}^{-1}(\x_{A_2}f^\circ_{s,A_1})^\circ_{s,A_2} +\x_{A_1}^{-1}\left(\x_{A_2}^{-1}\left((\x_{A_2}(\x_{A_1}f)^\circ_{s,A_1})^\circ_{s,A_2}\right) \right)\\
&=(f^\circ_{s,A_1})^\circ_{s,A_2}+J_1(-\beta_1^{-1}(\x_{A_1}f)^\circ_{s,A_1})^\circ_{s,A_2}
+J_2(-\beta_2^{-1}\x_{A_2}f^\circ_{s,A_1})^\circ_{s,A_2}\\
&\quad +J_1\left(J_2\beta_1^{-1}\beta_2^{-1}(\x_{A_2}(\x_{A_1}f)^\circ_{s,A_1})^\circ_{s,A_2} \right)
=:F_{\emptyset}+J_1F_{1}+J_2F_{2}+J_1(J_2F_{12}).
\end{align*}
The functions $F_{1},F_{2},F_{12}$ extend continuously to the points in $\OO$ with $\beta_1=0$ or $\beta_2=0$ (see Remark \ref{rem:sliceness_continuity_A}). 
\end{proof}

Given a partition $\PP=\{A_1,\ldots,A_\ell\}$ of $[n]$, 
a set $\OO\subseteq M$ invariant w.r.t.\ the reflection $r_i$ for every $i=1,\ldots,n$, and a function $f:\OO\to\Aa$, let us define recursively functions $\SP^m_K(f):\OO\to\Aa$, for $m=0,\ldots, \ell$ and $K\subseteq[m]$ as follows:
\begin{equation}\label{def:sp}
\begin{cases}
{}\SP^0_\emptyset(f):=f&\text{\quad for }m=0,
\\\SP^m_K(f):=\left(\x_{A_m}^{\mathbf{1}_K(m)}\SP^{m-1}_{K\setminus\{m\}}(f)\right)^\circ_{s,A_m} &\text{\quad for }m=1,\ldots,\ell.
\end{cases}
\end{equation}
where $\mathbf{1}_{K}$ is the characteristic function of $K$. With this notation, the functions $F_\emptyset,F_{1},F_{2},F_{12}$ in the previous proposition can be expressed as
\[\textstyle
F_K=\prod_{i\in K}(-\beta_i)^{-1}\SP^2_K(f).
\]
Thanks to Remark \ref{rem:sliceness_continuity_A}, if $f\in C^0(\OO,\Aa)$, then every function $F_K=\prod_{i\in K}(-\beta_i)^{-1}\SP^m_K(f)$ has a continuous extension to $\OO$. Moreover, for every $m=0,\ldots,\ell-1$ and $K\subseteq[m]$, it holds
\begin{equation}\label{eq:SP}
\SP^m_K(f)=\SP^{m+1}_{K}(f)+ \x_{A_{m+1}}^{-1}\SP^{m+1}_{K\cup\{m+1\}}(f).  
\end{equation}


\begin{definition}\label{def:Ja}
Given $J = (J_1, \ldots, J_\ell) \in \mathbb{S}_{\mathbb{A}}^\ell$, a sequence ${\bf h}=(h_1,\ldots,h_m)$, with $m\le\ell\le n$ and $h_1,\ldots,h_m\in[\ell]$ and $a\in\Aa$, we set
\[
[J,a]_{\bf h}=J_{h_1}(J_{h_2}(\cdots (J_{h_{m-1}}(J_{h_m}a))\cdots)).
\]
If $K=\{k_1,\ldots,k_m\}\subseteq[\ell]$, with $k_1<\cdots <k_m$, we denote by $[J,a]_K$ the product
\[
[J,a]_K=[J,a]_{(k_1,\ldots,k_m)}=J_{k_1}(J_{k_2}(\cdots (J_{k_{m-1}}(J_{k_m}a))\cdots))
\]
 (see \cite[Definition 8.1]{GhiloniStoppato_arXiv24}).
\end{definition}

\begin{lemma}\label{lem:lemmaJa}
If $J_{h_1},\ldots, J_{h_m}$ anticommute, then for any permutation $\sigma$ of $\{1,\ldots,m\}$, it holds $[J,a]_{\bf h}=\epsilon_\sigma[J,a]_{\sigma({\bf h})}$, with $\epsilon_\sigma$ denoting the sign of $\sigma$. 
Moreover, if $J_{k_1}, \ldots, J_{k_m}$ anticommute and $i\in K$, then $[J, a]_K = \epsilon(i, K) J_i [J, a]_{K \setminus \{i\}}$ for an appropriately defined sign $\epsilon(i,K)\in \{\pm 1\}$. 
\end{lemma}
\begin{proof}
The first property follows from the alternativity of the associator in $\Aa$. 
When $m=2$ it holds
\begin{align*}
0&=[J_{h_1},J_{h_2},a]+[J_{h_2},J_{h_1},a]=(J_{h_1}J_{h_2})a-J_{h_1}(J_{h_2}a)+(J_{h_2}J_{h_1})a-J_{h_2}(J_{h_1}a)\\
&=-J_{h_1}(J_{h_2}a)-J_{h_2}(J_{h_1}a)=-[J,a]_{(h_1,h_2)}-[J,a]_{(h_2,h_1)}.
\end{align*}
The general case follows easily from this. The second statement is an immediate consequence of the first.
\end{proof}

Note that the anti-commutation hypothesis is automatically true if, for some list of steps $T$ of length $\tau$, we have $\ell = \tau$ and $J = (J_1, \ldots, J_\tau)$ belonging to the $T$-torus $\mathbb{T} :=\prod_{i = 1}^\tau \mathbb{S}_i$, where $\s_i:=\s_{\Aa}\cap M_{A_{i}}$ is the unit sphere in $M_{A_{i}}\cap\ker(t)$ (see \cite[Definition 3.1]{GhiloniStoppatoJGP}).  







\begin{theorem}\label{teo:kerSP}
Let $\PP=\{A_1,\ldots,A_\ell\}$ be a partition of the set $[n]$. 
Let $\OO=\OO_D$ be a $\PP$-circular open subset of $M$. 
If $f\in\ker\mathscr S_\PP$, then there exists a finite sequence $\{ F_K \}_{K \subseteq [n]}$ in $C^0(D,A)$ such
that for every $x\in\OO$ it holds
\begin{equation}\label{eq:fJF}
f(x)=\sum_{K\subseteq[\ell]}[J,F_K(x_0,\beta)]_K.   
\end{equation}
Here $(x_0,\beta)=(x_0,\beta_1,\ldots,\beta_\ell)\in D$ and $x=x_0+\sum_{i=1}^\ell\beta_i J_i\in\OO$. The units $J_i\in \s_M$ are defined as $J_i:=\beta_i^{-1}\x_{A_i}$ if $|\beta_i|=\|\x_{A_i}\|\ne0$, and $J_i$ equal to any element of $\s_{M_{A_i}}$ if $\beta_i=0$, $\x_{A_i}=0$.
The function with components $F_K$ 
satisfies 
\begin{equation}\label{eq:Tstem}
F_K(x_0,\bar \beta^h)=(-1)^{|K\cap\{h\}|}F_K(x_0,\beta)
\end{equation}
for $h=1,\ldots,\ell$, $K\subseteq[\ell]$. 
Conversely, if $f(x)=\sum_{K\subseteq[\ell]}[J,F_K(x_0,\beta)]_K$ with $F_K$ satisfying \eqref{eq:Tstem}, then 
$\mathscr S_\PP f=0$ on $\OO$.
\end{theorem}

\begin{proof}
We can iterate the procedure described in the proof of Proposition \ref{pro:kerS2}, using the fact that for all $m=0\ldots,\ell-1$ and $K\subseteq[m]$, the function $\SP^m_K(f)$ belongs to $\ker\mathscr S_{A_{m+1}}$. After setting
\[\textstyle
F_K:=\prod_{i\in K}(-\beta_i)^{-1}\SP^m_K(f),
\]
and using \eqref{eq:SP}, we find that at every step $m$ we have
\[
f(x)=\sum_{K\subseteq[m+1]} \x_{A_{k_1}}^{-1}(\x_{A_{k_2}}^{-1}(\cdots(\x_{A_{k_{m+1}}}^{-1}\SP^{m+1}_{K\cup\{m+1\}}(f)(x))\cdots))=\sum_{K\subseteq[m+1]}[J,F_K(x_0,\beta)]_K.
\]
When $m=\ell-1$ we get \eqref{eq:fJF}. 
The (anti)symmetric properties of the functions $F_K$ w.r.t.\ $\beta_1,\ldots,\beta_\ell$ follow immediately from their definition. 

Conversely, if equation \eqref{eq:fJF} holds, then for any $i=1,\ldots,\ell$ one can write, using Lemma \ref{lem:lemmaJa},
\[
f(x) = \sum_{[\ell] \supseteq K \not\ni i} [J, F_K(x_0,\beta)]_K + J_i \sum_{[\ell] \supseteq K \ni i} \epsilon(i, K) [J, F_K(x_0,\beta)]_{K \setminus \{i\}} =: \widetilde{F}_\emptyset + J_i \widetilde{F}_1,
\]
with $\widetilde F_\emptyset,\widetilde F_1$ an even/odd pair w.r.t.\ $\beta_i$. Proposition \ref{pro:SA} gives $\mathscr S_{A_i}f=0$. 
Therefore $\mathscr S_\PP f=0$ on $\OO$.
\end{proof}

\begin{definition}\label{def:pslice}
Let $\PP=\{A_1,\ldots,A_\ell\}$ be a partition of the set $[n]$. 
Let $\OO=\OO_D$ be a $\PP$-circular open subset of $M$. For every $K\subseteq[\ell]$, let  $F_K:D\subset\rr^{\ell+1}\to\Aa$ be a function which satisfies condition \eqref{eq:Tstem}. 
A function $f:\OO\to\Aa$ such that $f(x)=\sum_{K\subseteq[\ell]}[J,F_K(x_0,\beta)]_K$ for every $x=x_0+\sum_{i=1}^\ell\beta_i J_i\in\OO$ will be called the \emph{$\PP$-slice function} induced on $\OO$ by the \emph{$\PP$-stem function} $F=(F_K)_{K\subseteq[\ell]}$. 
\end{definition}

In view of Theorem \ref{teo:kerSP}, a $C^1$ function $f$ is $\PP$-slice if and only if it belongs to $\ker\mathscr S_\PP$. Observe that every function $f:\OO\to\Aa$ is $\PP$-slice for $\PP=\{\{1\},\ldots,\{n\}\}$.

\begin{proof}[Proof of Theorem \ref{teo:Tregular}]
Let $\PP=\{A_1,\ldots,A_\ell\}$ with $\ell=t_0+\tau$ and 
\[
A_1=\{t_0+1,\ldots,t_1\},\ldots,A_{\tau}=\{t_{\tau-1}+1,\ldots,t_\tau\},
A_{\tau+1}=\{1\},\ldots,A_{\tau+t_0}=\{t_0\},
\]
where the list ends with $A_{\tau}$ if $t_0=0$.
For every $i=1,\ldots,\tau$, let $\s_i:=\s_{\Aa}\cap M_{A_{i}}$ be the unit sphere in $M_{A_{i}}\cap\ker(t)$. 
Let $J=(J_1,\ldots,J_\tau)\in\mathbb{T}.$
Consider the hypercomplex subspace $M_J=\langle 1,v_1,\ldots,v_{t_0},J_1,\ldots,J_\tau\rangle$, with real coordinates $(x_0,x_1,\ldots,x_{t_0},\beta_1,\ldots,\beta_\tau)$ w.r.t.\ the basis $\B_J=(1,v_1,\ldots,v_{t_0},J_1,\ldots,J_\tau)$. 
The Cauchy-Riemann operator of $M_J$ is
\[
\dibar_{M_J}=\partial_{x_0}+\sum_{j=1}^{t_0}v_j\partial_{x_j}+\sum_{i=1}^\tau J_i\partial_{\beta_i}=\dibar_J,
\]
where the notation $\dibar_J$ is the one adopted in \cite[\S3]{GhiloniStoppatoJGP}. A function $f$ is strongly $T$-regular in $\OO$ if and only if for every $J=(J_1,\ldots,J_\tau)\in \mathbb{T}$, the restriction of $f$ to $\OO\cap M_J$ is monogenic, namely, $\dibar_Jf(x_0+\sum_{j=1}^{t_0}v_jx_j+\sum_{i=1}^\tau J_i\beta_i)=0$ on $\OO\cap M_J$.  

Given $f\in\F_\PP(\OO)$, we have, in particular, $\underline S_{A_{1}}f=\cdots =\underline S_{A_{\tau}}f=0$. From \eqref{eq:SA} we get, for every $i=1,\ldots,\tau$, that
\begin{equation}\label{eq:DD_P}
\DD_{A_{i}}f=-\x_{A_{i}}^{-1}\Eu_{A_{i}} f=\tfrac{\x_{A_{i}}}{\|\x_{A_{i}}\|}\sum_{j\in A_i}\tfrac{x_j}{\|\x_{A_i}\|}\partial_{x_j}f=J_i\partial_{\beta_i}f,  
\end{equation}
when $\x_{A_{i}}\ne0$ and $J_i=\text{sign}(\beta_i)\tfrac{\x_{A_{i}}}{\|\x_{A_{i}}\|}\in\s_i$. 
Since $\DD_{A_{\tau+j}}=v_{j}\partial_{x_{j}}$ for $j=1,\ldots,t_0$, it holds 
\[
0=D_{\PP}f=\bigg(\partial_{x_0}+\sum_{j=1}^{t_0}v_j\partial_{x_j}+\sum_{i=1}^{\tau}\DD_{A_{i}}\bigg)f=\dibar_J f
\]
for every $J\in\mathbb{T}$, and $f$ is strongly $T$-regular. 

Since $\mathscr S_{\{\tau+j\}}=0$ for every $j=1,\ldots,t_0$, the condition $\mathscr S_{\PP}f=0$ is equivalent to $f\in\ker \mathscr S_{A_1}\cap\cdots\cap\ker \mathscr S_{A_\tau}$. 
If in the iterative process described in the proof of Theorem \ref{teo:kerSP} we stop at step $m=\tau-1$, we obtain the decomposition
\[
f(x)=\sum_{H\subseteq[\tau] }[J, F_H(x_0,x_1,\ldots,x_{t_0},\beta_1,\ldots,\beta_\tau)]_H,  
\]
with functions $F_H(x_0,x_1,\ldots,x_{t_0},\beta_1,\ldots,\beta_\tau):=\prod_{i\in H}(-\beta_i)^{-1}\SP^{\tau}_H(f)(x)$ satisfying 
\[
F_H(x_0,x_1,\ldots,x_{t_0},\beta_1,\ldots,-\beta_h,\ldots,\beta_\tau)=(-1)^{|H\cap\{h\}|} F_H(x_0,x_1,\ldots,x_{t_0},\beta_1,\ldots,\beta_h,\ldots,\beta_\tau)
\]
for $h=1,\ldots,\tau$, $H\subseteq[\tau]$, showing that $f$ is a $T$-function. 

Conversely, if $f$ is a 
strongly $T$-regular function of class $C^1$ in $\OO$, then 
Theorem \ref{teo:kerSP} yields that $f\in\ker \mathscr S_\PP$. 
Moreover, $\DD_{A_i}f=-\x_{A_{i}}^{-1}\Eu_{A_{i}} f=J_i\partial_{\beta_i}f$ for every $J_i\in\s_i$ whence, at $x=x_0+\sum_{j=1}^{t_0}v_jx_j+\sum_{i=1}^\tau\beta_i J_i$ it holds
\[
D_{\PP}f(x)=\bigg(\partial_{x_0}+\sum_{j=1}^{t_0}v_j\partial_{x_j}+\sum_{i=1}^{\tau}\DD_{A_{i}}\bigg)f(x)=\dibar_J f=0
\]
for every $J\in\mathbb{T}$, that is, $D_{\PP}f=0$ on $\OO$ and $f\in\F_\PP(\OO)$. 
\end{proof}

\begin{example}
Let $M$ be the eight-dimensional hypercomplex subspace 
\[
M=\langle1,e_1,e_2,e_3,e_4,e_5,e_6,e_{123456}\rangle
\]
of $\R_6$. Let $\PP=\{A_1,A_2,A_3\}$, with $A_1=\{1\}$, $A_2=\{2,3,4\}$, $A_3=\{5,6,7\}$. Consider the polynomial $f=P_{(1,2,2)}\in\F_\PP(M)$ obtained in \cite[Example 44]{PerottiFueterTrees} as the $\PP$-Dunkl-regular Cauchy-Kovalevskaya extension of the function $\x_{A_1}\x_{A_2}^2\x_{A_3}^2$:
\begin{align*}
f&=\tfrac{1}{15}x_0^5+
\tfrac{1}3x_0^4 \x_{A_1}-
\tfrac{1}3x_0^3\left(2\x_{A_1}\big(\x_{A_2}+\x_{A_3}\big)+\big(\|\x_{A_2}\|^2+\|\x_{A_3}\|^2\big)\right)-
x_0^2 \x_{A_1}\big(\|\x_{A_2}\|^2+\|\x_{A_3}\|^2\big)\\
&\quad +x_0\left(2\x_{A_1}\big(\x_{A_2}\|\x_{A_3}\|^2+\x_{A_3}\|\x_{A_2}\|^2\big)+\|\x_{A_2}\|^2\|\x_{A_3}\|^2 \right)
+\x_{A_1} \|\x_{A_2}\|^2\|\x_{A_3}\|^2.
\end{align*}
Then $f$ is a strongly $T$-function on $M$, where $T$ is the list of steps $(1, 4, 7)$. 
\end{example}






\begin{thebibliography}{10}

\bibitem{binosi2025dunklapproachsliceregular}
G.~Binosi, H.~De~Bie, and P.~Lian.
\newblock Dunkl approach to slice regular functions.
\newblock {\em Ann. Mat. Pura Appl. (4)}, 2025.

\bibitem{BDS}
F.~Brackx, R.~Delanghe, and F.~Sommen.
\newblock {\em Clifford analysis}, volume~76 of {\em Research Notes in
  Mathematics}.
\newblock Pitman (Advanced Publishing Program), Boston, MA, 1982.

\bibitem{Ren_et_al}
P.~Cerejeiras, U.~K\"ahler, and G.~Ren.
\newblock Clifford analysis for finite reflection groups.
\newblock {\em Complex Var. Elliptic Equ.}, 51(5-6):487--495, 2006.

\bibitem{CoSaSt2009Israel}
F.~Colombo, I.~Sabadini, and D.~C. Struppa.
\newblock Slice monogenic functions.
\newblock {\em Israel J. Math.}, 171:385--403, 2009.

\bibitem{DeBieGenestVinet}
H.~De~Bie, V.~X. Genest, and L.~Vinet.
\newblock The {$\mathbb{Z}_2^n$} {D}irac-{D}unkl operator and a higher rank
  {B}annai-{I}to algebra.
\newblock {\em Adv. Math.}, 303:390--414, 2016.

\bibitem{deBieOrstedSombergSoucek}
H.~De~Bie, B.~{\O}rsted, P.~Somberg, and V.~Sou\v cek.
\newblock Dunkl operators and a family of realizations of
  {$\mathfrak{osp}(1|2)$}.
\newblock {\em Trans. Amer. Math. Soc.}, 364(7):3875--3902, 2012.

\bibitem{DentoniSce}
P.~Dentoni and M.~Sce.
\newblock Funzioni regolari nell'algebra di {C}ayley.
\newblock {\em Rend. Sem. Mat. Univ. Padova}, 50:251--267 (1974), 1973.

\bibitem{Dunkl}
C.~F. Dunkl.
\newblock Differential-difference operators associated to reflection groups.
\newblock {\em Trans. Amer. Math. Soc.}, 311(1):167--183, 1989.

\bibitem{DunklXu}
C.~F. Dunkl and Y.~Xu.
\newblock {\em Orthogonal polynomials of several variables}, volume~81 of {\em
  Encyclopedia of Mathematics and its Applications}.
\newblock Cambridge University Press, Cambridge, 2001.

\bibitem{GeStoSt2013IIed}
G.~Gentili, C.~Stoppato, and D.~C. Struppa.
\newblock {\em Regular functions of a quaternionic variable}.
\newblock Springer Monographs in Mathematics. Springer, Cham, second edition,
  2022.

\bibitem{GeSt2006CR}
G.~Gentili and D.~C. Struppa.
\newblock A new approach to {C}ullen-regular functions of a quaternionic
  variable.
\newblock {\em C. R. Math. Acad. Sci. Paris}, 342(10):741--744, 2006.

\bibitem{GeSt2007Adv}
G.~Gentili and D.~C. Struppa.
\newblock A new theory of regular functions of a quaternionic variable.
\newblock {\em Adv. Math.}, 216(1):279--301, 2007.

\bibitem{GeStRocky}
G.~Gentili and D.~C. Struppa.
\newblock Regular functions on the space of {C}ayley numbers.
\newblock {\em Rocky Mountain J. Math.}, 40(1):225--241, 2010.

\bibitem{SliceFueterRegular}
R.~Ghiloni.
\newblock Slice {F}ueter-regular functions.
\newblock {\em J. Geom. Anal.}, 31(12):11988--12033, 2021.

\bibitem{AIM2011}
R.~Ghiloni and A.~Perotti.
\newblock Slice regular functions on real alternative algebras.
\newblock {\em Adv. Math.}, 226(2):1662--1691, 2011.

\bibitem{VolumeCauchy}
R.~Ghiloni and A.~Perotti.
\newblock Volume {C}auchy formulas for slice functions on real associative
  *-algebras.
\newblock {\em Complex Var. Elliptic Equ.}, 58(12):1701--1714, 2013.

\bibitem{Gh_Pe_GlobDiff}
R.~Ghiloni and A.~Perotti.
\newblock Global differential equations for slice regular functions.
\newblock {\em Math. Nachr.}, 287(5-6):561--573, 2014.

\bibitem{AlgebraSliceFunctions}
R.~Ghiloni, A.~Perotti, and C.~Stoppato.
\newblock The algebra of slice functions.
\newblock {\em Trans. Amer. Math. Soc.}, 369(7):4725--4762, 2017.

\bibitem{GhiloniStoppatoJGP}
R.~Ghiloni and C.~Stoppato.
\newblock A unified notion of regularity in one hypercomplex variable.
\newblock {\em J. Geom. Phys.}, 202:Paper No. 105219, 13, 2024.

\bibitem{GhiloniStoppato_arXiv24}
R.~Ghiloni and C.~Stoppato.
\newblock A unified theory of regular functions of a hypercomplex variable.
\newblock {\em Bull. Sci. Math.}, 209:Paper No. 103794, 92, 2026.

\bibitem{GhiloniStoppato_JMAA26}
R.~Ghiloni and C.~Stoppato.
 \newblock A manifold Fueter-Sce phenomenon in one hypercomplex variable.
 \newblock {\em  J. Math. Anal. Appl.}, 564:130903, 2026.

\bibitem{GHS}
K.~G{\"u}rlebeck, K.~Habetha, and W.~Spr{\"o}{\ss}ig.
\newblock {\em Holomorphic functions in the plane and {$n$}-dimensional space}.
\newblock Birkh\"auser Verlag, Basel, 2008.

\bibitem{Huo_Ren_Xu_arXiv25}
Q.~Huo, G.~Ren, and Z.~Xu.
\newblock Monogenic functions over real alternative *-algebras: fundamental
  results and applications, \url{https://arxiv.org/abs/2504.01359}, 2025.

\bibitem{JinRenSabadini2020}
M.~Jin, G.~Ren, and I.~Sabadini.
\newblock Slice {D}irac operator over octonions.
\newblock {\em Israel J. Math.}, 240(1):315--344, 2020.

\bibitem{MeyerBook}
C.~Meyer.
\newblock {\em Matrix analysis and applied linear algebra}.
\newblock Society for Industrial and Applied Mathematics (SIAM), Philadelphia,
  PA, 2000.
\newblock With 1 CD-ROM (Windows, Macintosh and UNIX) and a solutions manual
  (iv+171 pp.).

\bibitem{MT1931}
G.~C. {Moisil} and N.~{Theodoresco}.
\newblock {Fonctions holomorphes dans l'espace}.
\newblock {\em {Mathematica, Cluj}}, 5:142--159, 1931.

\bibitem{Orsted}
B.~{\O}rsted, P.~Somberg, and V.~Sou\v cek.
\newblock The {H}owe duality for the {D}unkl version of the {D}irac operator.
\newblock {\em Adv. Appl. Clifford Algebr.}, 19(2):403--415, 2009.

\bibitem{CRoperators}
A.~Perotti.
\newblock Cauchy-{R}iemann operators and local slice analysis over real
  alternative algebras.
\newblock {\em J. Math. Anal. Appl.}, 516(1):Paper No. 126480, 34, 2022.

\bibitem{PerottiFueterTrees}
A.~Perotti.
\newblock Fueter trees for {D}unkl-regular functions over alternative *-algebras, 
\url{https://arxiv.org/abs/2604.12876}, 2026.

\bibitem{Rosler}
M.~R\"osler.
\newblock Dunkl operators: theory and applications.
\newblock In {\em Orthogonal polynomials and special functions ({L}euven,
  2002)}, volume 1817 of {\em Lecture Notes in Math.}, pages 93--135. Springer,
  Berlin, 2003.

\bibitem{Struppa2015Algebras}
D.~C. Struppa.
\newblock {\em Slice Hyperholomorphic Functions with Values in Some Real
  Algebras}, pages 1631--1650.
\newblock Springer Basel, Basel, 2015.

\bibitem{Sabadini_Xu_TAMS}
Z.~Xu and I.~Sabadini.
\newblock Generalized partial-slice monogenic functions.
\newblock {\em Trans. Amer. Math. Soc.}, 378(2):851--883, 2025.

\bibitem{Sabadini_Xu_Octo_TAMS}
Z.~Xu and I.~Sabadini.
\newblock Generalized partial-slice monogenic functions: the octonionic case.
\newblock {\em Trans. Amer. Math. Soc.}, 379(5):3257--3291, 2026.

\end{thebibliography}

\end{document}